\documentclass{amsart}
\usepackage{marktext}
\usepackage{gimac}
\usepackage{caption}
\usepackage{accents}
\newcommand{\ubar}[1]{\underaccent{\bar}{#1}}

\colorlet{darkgreen}{green!50!black}
\usepackage{subcaption}

\DeclareMathOperator{\pe}{Pe}

\newcommand{\cB}{\mathcal{B}}

\newcommand{\cF}{\mathcal{F}}

\newcommand{\cH}{\mathcal{B}}

\newcommand{\curE}{\mathscr E}
\newcommand{\U}{\mathscr U}

\newcommand\dist{\mathrm{dist}}
\definecolor{applegreen}{rgb}{0.55, 0.71, 0.0}
\definecolor{applered}{RGB}{225, 160, 160}
\definecolor{bluebell}{RGB}{160, 160, 228}

\begin{document}
\title{Bounds on the heat transfer rate via passive advection}

\author[Iyer]{Gautam Iyer}
\address{%
  Department of Mathematical Sciences,
  Carnegie Mellon University,
  Pittsburgh, PA 15213.}
\email{gautam@math.cmu.edu}

\author[Van]{Truong-Son Van}
\address{%
  Department of Mathematics,
  University of Pennsylvania,
  Philadelphia, PA 19104.}
\email{tsvan@sas.upenn.edu}

\begin{abstract} 
    In heat exchangers, an incompressible fluid is heated initially and cooled at the boundary. The goal is to transfer the heat to the boundary as efficiently as possible. In this paper we study a related steady version of this problem where a steadily stirred fluid is uniformly heated in the interior and cooled on the boundary. For a given large P\'eclet number, how should one stir to minimize some norm of the temperature? This version of the problem was previously studied by Marcotte, Doering et\ al.\ (SIAM Appl.\ Math '18) in a disk, where the authors used matched asymptotics to show that when the P\'eclet number, $\pe$, is sufficiently large one can stir the fluid in a manner that ensures the total heat is $O(1/\pe)$. In this paper we confirm their results with rigorous proofs, and also provide an almost matching lower bound. For simplicity, we work on the infinite strip instead of the unit disk and the proof uses probabilistic techniques. 
\end{abstract}
\thanks{%
  This work has been partially supported by
  the National Science Foundation under grants
  DMS-1814147, %
   1812609,
  and the Center for Nonlinear Analysis.
}
\dedicatory{Dedicated to Charlie Doering, whose life and work was an inspiration to us.}

\maketitle

\section{Introduction}\label{s:intro}

A heat exchanger is a system used to transfer heat between a fluid and a heat source or sink, for either heating or cooling.
These are used for both heating and cooling processes and have a broad range applications including combustion engines, sewage treatment, nuclear power plants and cooling CPUs in personal computers~\cite{WeisbergBauEA92,QuMudawar02,VallaghePatera14,SheikholeslamiHaqEA19,AlamKim18,MarcotteDoeringEA18,WangWangEA18,DoeringTobasco19,LienhardLienhard20}.

The temperature of the fluid in the heat exchanger evolves according to the advection diffusion equation
\begin{equation}\label{e:theta}
  \partial_t \theta + v \cdot \grad \theta - \kappa \lap \theta = 0 \qquad \text{in } \Omega \,,
\end{equation}
where $\Omega \subseteq \R^d$ is the region occupied by the fluid.
Here $\theta$ is the temperature of the fluid, $\kappa$ is the thermal diffusivity and $v = v(x, t)$ is velocity field of the fluid.
Throughout this paper we will assume the fluid is incompressible and doesn't flow through the container walls.
That is, we require
\begin{equation}\label{e:incomp}
\dv v = 0 \quad\text{in } \Omega\,,
\qquad\text{and}\qquad
v \cdot \hat n = 0 \quad\text{on } \partial \Omega \,.
\end{equation}
Some portion of the boundary of $\Omega$ may be insulated, and some portion may be connected to a heat source/sink maintained at a constant temperature.
Denoting these pieces by $\partial_N \Omega$ and $\partial_D \Omega$ respectively, and normalizing so that the temperature of the heat source/sink is $0$, we study~\eqref{e:theta} with mixed Dirichlet/Neumann boundary conditions
\begin{equation*}
  \partial_{\hat n} \theta = 0 \quad \text{on } \partial_N \Omega\,,
  \qquad\text{and}\qquad 
  \theta = 0 \quad\text{on } \partial_D \Omega\,.
\end{equation*}

A problem of practical interest is to minimize some norm of the temperature under a constraint on the stirring velocity field.
Note, here we assume~\eqref{e:theta} is a passive scalar equation -- the velocity field~$v$ is prescribed and is not coupled to the temperature profile.
The active scalar case entails coupling $v$ to $\theta$ via the Boussinesq system and leads to Rayleigh--B\'enard convection which has been extensively studied~\cite{Rayleigh16,SolomonGollub88,Kadanoff01,DoeringOttoEA06}.

In order to simplify matters, we set $\kappa = \frac{1}{2}$, assume $v$ is time independent, and assume the initial temperature~$\theta_0$ is identically~$1$.
In this case we note that
\begin{equation*}
  T \defeq \int_0^\infty \theta(x, t) \, dt
\end{equation*}
satisfies the Poisson problem
\begin{equation}\label{e:T}
  -\frac{1}{2} \lap T + v \cdot \grad T = 1\,,
\end{equation}
in  $\Omega$, with boundary conditions 
\begin{equation}\label{e:TBC}
  T = 0 \quad\text{on } \partial_D \Omega\,,
  \qquad\text{and}\qquad
  \partial_{\hat n} T = 0 \quad\text{on } \partial_N \Omega\,.
\end{equation}
Now a simplified optimization problem of interest is minimize a norm of $T$ under a constraint on the advecting velocity field~$v$.

In the recent paper~\cite{MarcotteDoeringEA18}, the authors studied this minimization problem when~$\Omega \subseteq \R^2$ is a disk of radius~$1$, and $\partial_N \Omega = \emptyset$.
Given $p \in [1,\infty)$ and $\U > 0$, let $\mathcal V^{k,p}_{\U}$ be the set of all $W^{k,p}$ velocity fields satisfying~\eqref{e:incomp} such that
\begin{equation}\label{e:vConstraint}
  \norm{v}_{W^{k,p}(\Omega)} \leq \U\,,
\end{equation}
and define
\begin{equation*}
  \mathcal E^{k,p}_{q}(\U) \defeq
  \inf_{v \in \mathcal V^{k,p}_{\U}} \norm{T^v}_{L^q} \,.
\end{equation*}
Here $T^v$ is simply the solution to~\eqref{e:T}--\eqref{e:TBC}, and we introduced the superscript~$v$ to emphasize the dependence of~$T$ on~$v$.

Physically when~$k=0$ and~$p = 2$, the constraint~\eqref{e:vConstraint} limits the kinetic energy of the ambient fluid.
If the domain~$\Omega$ has an associated length scale of order~$1$, the quantity $\U$ is the \emph{P\'eclet number} --- a non-dimensional ratio measuring the relative strength of the advection to the diffusion.
When the P\'eclet number is sufficiently large, the authors of~\cite{MarcotteDoeringEA18} use matched asymptotics to show
\begin{equation}\label{e:MCDTY}
  \mathcal E^{0,2}_{1}(\U) \leq O\paren[\Big]{ \frac{1}{\U}}\,,
\end{equation}
and support their results with numerics.

In this paper we revisit this problem and aim to provide mathematically rigorous proofs of the bounds in~\cite{MarcotteDoeringEA18}.
Making matched asymptotics rigorous arises in many situations and has been extensively studied (see for instance~\cite{BensoussanLionsEA78,Kushner84,Nguetseng89,Evans90,Allaire92,PavliotisStuart08}).
In this situation, however, the flow considered in~\cite{MarcotteDoeringEA18} leads to a degenerate homogenization problem, for which one can not use standard techniques.
Instead we reformulate the problem probabilistically and use asymmetric large deviations estimates to  handle the degenerate diffusivity.

\begin{figure}[htb]
  \centering%
  \begin{minipage}[t]{.5\linewidth}\centering%
    \includegraphics[width=.75\linewidth]{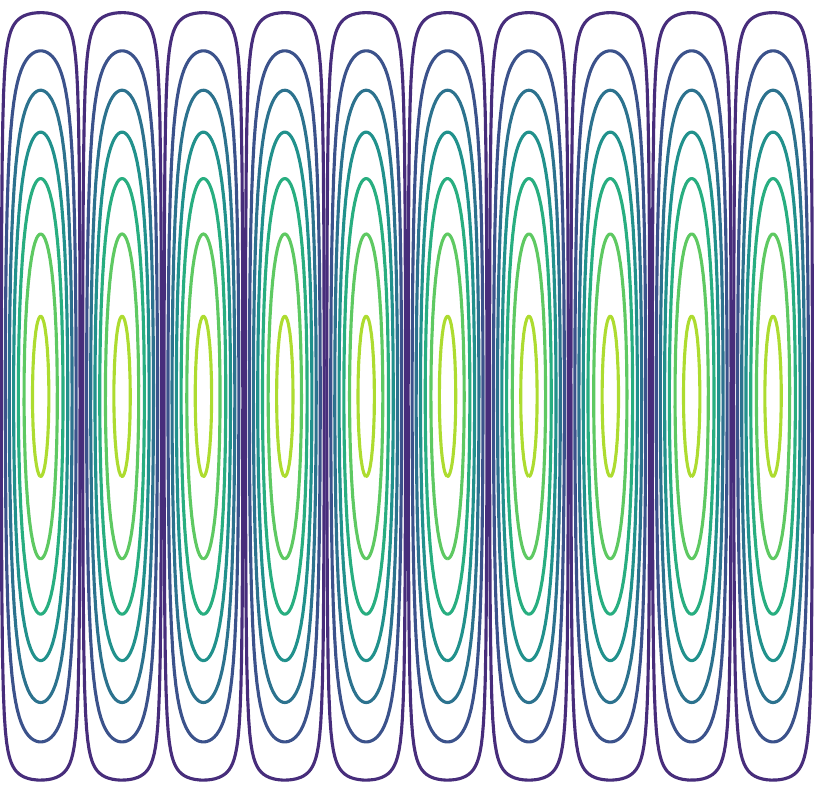}
    \caption{Tall and thin convection rolls}
    \label{f:convection}
  \end{minipage}%
  \begin{minipage}[t]{.5\linewidth}\centering%
    \includegraphics[width=.75\linewidth]{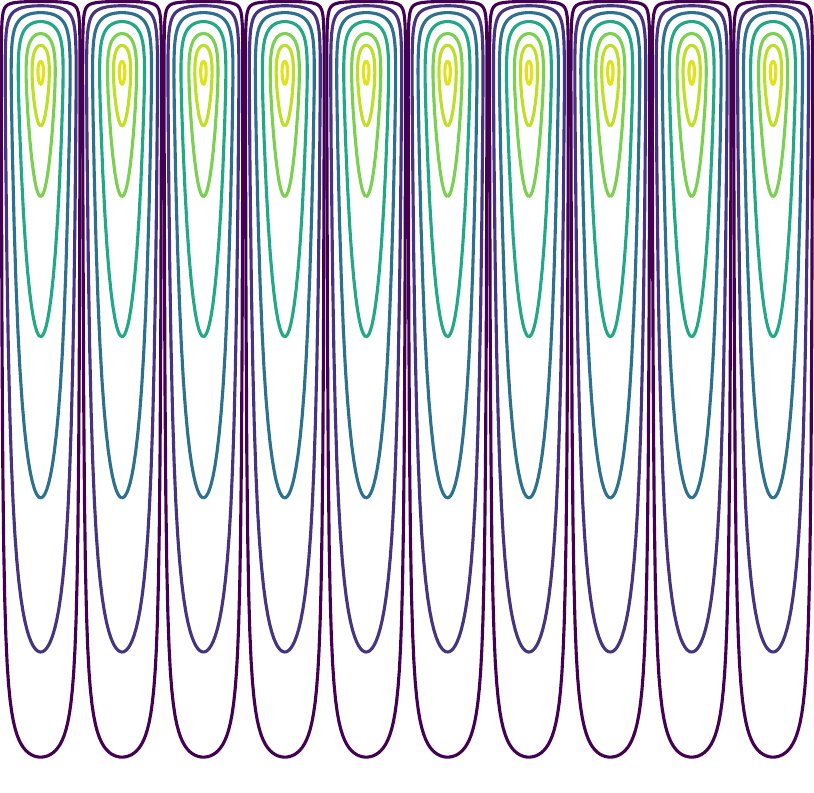}
    \caption{Skewed tall and thin convection rolls.}
    \label{f:skew}
  \end{minipage}%
\end{figure}
To simplify the proofs, we study the problem in a horizontal strip instead of the disk.
For boundary conditions we cool the top of the strip, insulate the bottom, and impose $2$-periodic boundary conditions in the horizontal direction.
To prove the upper bound~$\mathcal E^{0,p}_q(\U)$ we only need to find a velocity field~$v \in \mathcal V^{0,p}_\U$ for which $\norm{T^v}_{L^q} \leq O(1 /\U)$.
A natural first guess would be to choose a velocity field that forms many tall and thin convection rolls, with height $O(1)$, and width / amplitude that depend on the P\'eclet number.
This, however, turns out to be suboptimal, and yields a bound that is worse than~\eqref{e:MCDTY}.
To obtain the bound~\eqref{e:MCDTY} one needs to consider tall and thin convection rolls whose center is very close to the top of the strip.
This is the analogue of the velocity fields used in~\cite{MarcotteDoeringEA18}, and is shown in~\ref{f:skew}.

To formulate our result precisely, let $S = \R \times (0, 1) \subseteq \R^2$ be an infinite horizontal strip and $\partial_D S = \R \times \set{1}$ be the top boundary (where we impose homogeneous Dirichlet boundary conditions), and $\partial_N S = \R \times \set{0}$ the bottom boundary (where we impose homogeneous Neumann boundary conditions).
We will impose $2$-periodic boundary conditions in the horizontal direction and identify the function spaces $H^1(S)$ and  $L^2(S)$ with $H^1(\Omega)$ and $L^2(\Omega)$, respectively, where $\Omega \defeq (0, 2) \times (0, 1)$.

\begin{theorem}\label{t:energy}
  there exists a constant $C$ such that for $q\in [1,\infty]$,
  \begin{equation}\label{e:E1lowerBdLinf}
      \mathcal E^{0,\infty}_{q}(\U) 
      \geq  \frac{1}{C \U}\,.
  \end{equation}
  Furthermore, for every $\mu>0$, $p, q \in [1, \infty]$, we have
  \begin{equation}\label{e:E1upperBdLinf}
    \begin{dcases}
      \mathcal E^{0,p}_{q}(\U) 
    \leq \frac{C \ln \U}{\U} & p \in [1,2)\,, \\
      \mathcal E^{0,p}_{q}(\U) 
    \leq \frac{C_\mu \ln \U}{\U^{\frac{2 p}{3p - 2} - \mu}} & p \in [2,\infty] \,,
    \end{dcases}
  \end{equation}
  whenever the P\'eclet number, $\U$, is sufficiently large.
\end{theorem}

For $p, q < \infty$, upper bound in~\eqref{e:E1upperBdLinf} is suboptimal.
Indeed, forthcoming work of Doering and Tobasco uses methods in~\cite{DoeringTobasco19} to show that
\begin{equation}\label{e:E1upperBd2}
  \mathcal E^{0, p}_q(\U) \leq \frac{C}{\U}
  \qquad\text{for every } p, q \in [1, \infty)\,,
\end{equation}
and some constant $C = C(p, q)$ and all sufficiently large $\U$.
This is an improvement of~\eqref{e:E1upperBdLinf} by a logarithmic factor for $p \in [1, 2)$, and an arbitrarily small algebraic power for $p = 2$, and by a fixed algebraic power for $p \in (2, \infty)$.
For~$q = \infty$, however, the methods in~\cite{DoeringTobasco19} do not work.
In this case we believe that the logarithmic factor in~\eqref{e:E1upperBdLinf} is necessary due to the presence of hyperbolic critical points, but we are presently unable to prove this.

We do not presently know how to prove any lower bound for $\mathcal E^{0,p}_q(\U)$ when $p < \infty$.
For~$p = \infty$, however, we can use the Eikonal equation to obtain the lower bounded stated in~\eqref{e:E1upperBdLinf} in general domains.
We state this result next.
\begin{proposition}\label{p:TlowerGen}
  Let $d \geq 2$, and $\Omega \subseteq \R^d$ be a bounded domain with smooth boundary~$\partial \Omega$.
  Decompose the boundary as $\partial \Omega = \partial_D \Omega \cup \partial_N \Omega$, with $\partial_D \Omega \neq \emptyset$.
  Then
  \begin{equation}\label{e:E0infqLower}
    \mathcal E^{0, \infty}_q(\U) \geq \frac{1}{C \U}
    \qquad
    \text{for every } q \in [1, \infty]\,,
  \end{equation}
  for some constant $C= C(\Omega)$, and all sufficiently large~$\U$.
\end{proposition}
\begin{remark*}
  As we will see in the proof (specifically from inequality~\eqref{e:TvLower}, below), the constant $C=C(q, \Omega)$ can be computed in terms of the $L^q$ norm of the solution to the Eikonal equation in~$\Omega$.
\end{remark*}

Next we study the behavior of $\mathcal E^{1,p}_q(\mathscr E)$ when $\mathscr E$ is large.
Physically this corresponds to minimizing the $L^q$ norm of the steady state temperature~$T$ under an \emph{enstrophy} constraint on the stirring velocity field.
In this case it turns out that using standard convection rolls (as shown in Figure~\ref{f:convection}) yields a better upper bound on $\mathcal E^{1,p}_q(\mathscr E)$ than the skewed tall and thin rolls (as shown in Figure~\ref{f:skew}).
We note, however, that we have no matching lower bound and the skewed tall and thin convection rolls may not provide the optimal upper bound.
Indeed, the branched flows introduced recently by Doering and Tobasco~\cite{DoeringTobasco19} may provide the optimal bound in the enstrophy constrained case.
Unfortunately, due to their complicated geometry, they can not be analyzed by the techniques we use.
The best bound we can obtain is as follows.
\begin{proposition}\label{p:enstrophy}
  For every $p, q \in [1, \infty]$, there exists a finite constant $C = C(q)$ such that 
  \begin{equation}\label{e:E1upperBd}
      \mathcal E^{1,p}_{q}(\mathscr E)
      \leq \frac{C \abs{\ln \mathscr E}^{13}}{ \mathscr E^{2/5}  }  
  \end{equation}
  whenever $\mathscr E$ is sufficiently large.
  One velocity field that attains this upper bound uses convection rolls  with height $1$, width $\curE^{-1/5}$ and amplitude $\curE^{3/5}$ (see~Figure~\ref{f:convection}).
\end{proposition}

Even though there may be ``non-convection roll'' like flows that could improve the upper bound~\eqref{e:E1upperBd}, heuristics show that the bound~\eqref{e:E1upperBd} is the best one can achieve amongst the class of all ``convection roll'' like flows.
Moreover, for the tall and thin convection rolls used in proof of Proposition~\ref{p:enstrophy} one has matching upper and lower bounds on $\norm{T^v}_{L^\infty}$, up to a logarithmic factor.
Since such convection rolls arise in the study of magma flow in the Earth's mantle and various other contexts~\cite{TabataSuzuki02,KorenagaJordan03,GlasslHiltEA11,YangVerziccoEA15,OstillaMonico17}, the techniques used in the proof of Proposition~\ref{p:enstrophy} may be useful in some of these situations.

For a lower bound, clearly $\mathcal E^{1, \infty}_q(\mathscr E) \geq \mathcal E^{0,\infty}_q( \mathscr E)$, and hence by Proposition~\ref{p:TlowerGen} we have
\begin{equation*}
  \mathcal E^{1, \infty}_q(\mathscr E)
    \geq \frac{1}{C \mathscr E}\,,
    \qquad
    \text{for every } q \in [1, \infty]\,,
\end{equation*}
for all sufficiently large~$\mathcal E$.
We may be able to improve this by at most a logarithmic factor using a detailed analysis of the behavior near saddle points.
However, as mentioned earlier, we do not know whether the upper bound~\eqref{e:E1upperBd} is optimal and we are unable to obtain a matching lower bound.

\subsection*{Plan of the paper}
  In Section~\ref{sec:lower-optimal} we prove the lower bounds in Theorem~\ref{t:energy} and Proposition~\ref{p:TlowerGen}.
In Section~\ref{s:enstrophy},
we use an elementary scaling argument to reduce Proposition~\ref{p:enstrophy} to obtaining an upper bound on a degenerate cell problem (Proposition~\ref{p:TepBound}).
In Section~\ref{s:TepBound} we prove Proposition~\ref{p:TepBound} using probabilistic techniques, modulo two lemmas concerning exit from / the return to the boundary layer.
These lemmas are proved in Sections~\ref{s:bl} and~\ref{a:innerEst}.
The proofs of these lemmas rely on certain large deviations estimates which relegated to Appendix~\ref{s:tubelemmas}.
The proof of the upper bound in Theorem~\ref{t:energy} is similar to the proof of Proposition~\ref{p:TepBound}, and is presented in Section~\ref{sec:energy}.

\subsection*{Acknowledgements}
The authors thank Charlie Doering, Jean-Luc Thiffeault, Ian Tobasco and Noel Walkington for helpful discussions.

\section{Lower bounds}\label{sec:lower-optimal}

In this section we prove the lower bound in Theorem~\ref{t:energy} and the generalized version in Proposition~\ref{p:TlowerGen}.
The main idea in the proof is to consider an incompressible flow that moves directly towards the cold boundary.
Of course, this flow penetrates the boundary of the domain and so is not an element of~$\mathcal V^{0, \infty}_\U$.
However, it can still be used to build a sub-solution and prove the desired lower bound.
Since the proof in a strip is short and explicit, we present it first.

\begin{proof}[Proof of the lower bound in Theorem~\ref{t:energy}]
  Let $\ubar T$ be the solution to
  \begin{equation*}
    -\frac{1}{2} \partial_y^2  \ubar T - \U \partial_y \ubar T = 1
  \end{equation*}
  in the strip $S$ with $\ubar T = 0$ $\partial_D S$ and $\partial_y \ubar T = 0$ on $\partial_N S$.
  Explicitly solving this yields
  \begin{equation}\label{e:ubarT}
    \ubar T(y)
      = \frac{e^{-2\U}}{2 \U^2} \paren[\big]{1 - e^{2\U(1 -y)}}
  + \frac{1 - y}{\U}
  \end{equation}
  and hence $\partial_y \ubar T \leq 0$.

  We now claim that for any velocity field~$v$ such that $v_2 \geq -\U$, the function $\ubar T$ is a sub-solution to~\eqref{e:T}--\eqref{e:TBC}.
  Indeed,
  \begin{equation*}
    -\frac{1}{2} \lap \ubar T + v \cdot \grad \ubar T
      = -\frac{1}{2} \partial_y^2 \ubar T + v_2 \partial_y \ubar T
      \leq -\frac{1}{2} \partial_y^2 \ubar T + \U \partial_y \ubar T = 1\,.
  \end{equation*}
  The last inequality above followed from the fact that $v_2 \geq -\U$ and $\partial_y \ubar T \leq 0$.

  Thus by the comparison principle, for every $v \in \mathcal V^{0, \infty}_\U$ we must have $0 \leq \ubar T \leq T^v$.
  Hence $\norm{T^v}_{L^q} \geq \norm{\ubar T}_{L^q}$ and computing $\norm{\ubar T}_{L^q}$ using~\eqref{e:ubarT} yields the lower bound in~\eqref{e:E1upperBdLinf} as claimed.
\end{proof}

In general domains the sub-solution isn't as explicit and needs to be constructed using the Eikonal equation.

\begin{proof}[Proof of Proposition~\ref{p:TlowerGen}]
  Let $v \in L^\infty(\Omega)$, and $T = T^v$ be the solution of~\eqref{e:T}.
  For any $\epsilon > 0$ let $\tilde{T}^{\epsilon, \lambda}$ be the solution to the following viscous Hamilton-Jacobi equation
  \begin{equation*}
    \begin{dcases}
      \lambda \tilde T^{\epsilon,\lambda} -\epsilon \lap \tilde{T}^{\epsilon,\lambda} + \abs{\grad \tilde T^{\epsilon,\lambda}} = 1 \,, & x \in \Omega \,, \\
      \tilde T^{\epsilon,\lambda} = 0 \,,  & x \in \partial \Omega \,.
    \end{dcases}
  \end{equation*}
  Note that $\tilde T^{\epsilon,\lambda} \geq 0$ as $0$ is a subsolution to this equation.
  It is well known (see for instance~\cite{Calder18,Tran21}) that for every $\lambda > 0$, $\tilde T^{\epsilon,\lambda}$ converges uniformly as $\epsilon \to 0$ to the viscosity solution of the  equation
  \begin{equation*}
    \begin{dcases}
    \lambda \tilde T^{0, \lambda} + \abs{\grad \tilde T^{0, \lambda} } = 1 \,, & x\in \Omega \,, \\
    \tilde T^{0, \lambda} = 0 \,, & x \in \partial \Omega \,.
    \end{dcases}
  \end{equation*}
  Now letting $\lambda \to 0$, $\tilde T^{0,\lambda}$ converges uniformly to the viscosity solution of the Eikonal equation 
  \begin{equation*}
    \begin{dcases}
    \abs{\grad \tilde T^{0,0} } = 1 \,, & x\in \Omega \,, \\
    \tilde T^{0,0} = 0 \,, & x \in \partial \Omega \,.
    \end{dcases}
  \end{equation*}

  We claim that $\ubar{T}^{\epsilon,\lambda} \defeq \epsilon \tilde T^{\epsilon,\lambda}$ is a sub-solution of~\eqref{e:T} provided $\epsilon \leq 1/\norm{v}_{L^\infty}$.
  Indeed,
  \begin{multline*}
    -\lap \ubar{T}^{\epsilon,\lambda} + v \cdot \grad \ubar{T}^{\epsilon,\lambda}
      \leq -\lap \ubar{T}^{\epsilon, \lambda} + \epsilon \norm{v}_{L^\infty} \abs{\grad \tilde T^{\epsilon, \lambda}}
      \\
      \leq -\lap \ubar{T}^{\epsilon,\lambda} + \abs{\grad \tilde T^{\epsilon,\lambda}}
        + \frac{\lambda}{\epsilon} \ubar{T}^{\epsilon,\lambda}
      = -\epsilon \lap \tilde{T}^{\epsilon,\lambda} + \abs{\grad \tilde T^{\epsilon,\lambda}}
        + \lambda \tilde{T}^{\epsilon,\lambda}
      = 1\,.
  \end{multline*}
  Since $\ubar{T}^{\epsilon, \lambda} = 0$ on $\partial \Omega$, and $T^v$ is nonnegative, the minimum principle implies~$\ubar{T}^{\epsilon,\lambda} \leq T^v$ in~$\Omega$.
  This immediately implies
  \begin{equation*}
    \frac{1}{\epsilon} \norm{T^v}_{L^q}
      \geq \frac{1}{\epsilon} \norm{\ubar{T}^{\epsilon,\lambda}}_{L^q}
      \xrightarrow{\epsilon\to 0}
      \norm{\tilde T^{0, \lambda}}_{L^q}
      \xrightarrow{\lambda \to 0}
      \norm{\tilde T^{0, 0}}_{L^q}\,.
  \end{equation*}
  Thus when~$\epsilon$ is sufficiently small we have
  \begin{equation*}
    \norm{T^v}_{L^q} \geq \frac{\epsilon}{2} \norm{\tilde T^{0, 0}}_{L^q}\,.
  \end{equation*}
  Consequently, if $\norm{v}_{L^\infty}$ is sufficiently large, we can choose $\epsilon = \frac{1}{\norm{v}_{L^\infty}}$ and obtain
  \begin{equation}\label{e:TvLower}
    \norm{T^v}_{L^q} \geq \frac{1}{2 \norm{v}_{L^\infty}} \norm{\tilde T^{0, 0}}_{L^q}\,.
  \end{equation}
  This immediately implies the bound~\eqref{e:E0infqLower} as claimed.
\end{proof}
\section{Upper bound for enstrophy constrained convection rolls (Proposition~\ref{p:enstrophy})}\label{s:enstrophy}

Our aim in this section is to prove Proposition~\ref{p:enstrophy}.
First note that by doubling the domain and using symmetry and rescaling we can reduce the problem to proving~\eqref{e:E1upperBd} on the domain
\begin{equation*}
  S_2 \defeq \R \times (-1, 1) \,,
  \qquad\text{with}\qquad
  \partial_N S_2 = \emptyset\,,
  \qquad
  \partial_D S_2 = \R \times \set{-1, 1} \,,
\end{equation*}
and only using velocity fields~$v$ for which
\begin{equation}\label{e:vSymm}
  v_1(x_1, -x_2) = v_1(x_1, x_2)
  \qquad\text{and}\qquad
  v_2(x_1, -x_2) = -v_2(x_1, x_2)\,.
\end{equation}
We will now prove the upper bound~\eqref{e:E1upperBd} by producing a velocity field~$v$ (depending on~$\mathscr E$) such that we have
\begin{equation}\label{e:intTv}
  \norm{ T^v}_{L^\infty}  \leq 
  C \abs{\ln \curE}^{13}
     \paren[\Big]{ \frac{1}{ \curE  }  }^{2/5}  \,,
\end{equation}
for all $\mathscr E$ sufficiently large.
We do this by forming convection rolls with height~$1$, width $\epsilon$ and amplitude $A_\epsilon/\epsilon^2$ for some small $\epsilon$ and large $A_\epsilon$ (see Figure~\ref{f:convection}).
Moreover, as we will see shortly, 
$\epsilon$ and $A_\epsilon$ should be chosen according to
\begin{equation} \label{e:AEpsilonPe}
  \frac{A_\epsilon}{\epsilon^3} = \curE \,.
\end{equation}

To construct~$v$, consider a Hamiltonian $H \colon \R^2 \to \R$ such that
\begin{subequations}
\begin{align}
  \label{e:Ham1}
  H(x_1, -1) &= H(x_1, 1) = 0\,,
  \\
  \label{e:Ham2}
  H(x_1, -x_2) &= -H(x_1, x_2)\,,
  \\
  H( x_1 + 2, x_2 ) &= H( x_1, x_2)\,,
  \label{e:Ham3}
\end{align}
\end{subequations}
for all $(x_1, x_2) \in \R^2$.
To obtain convection rolls of width~$\epsilon$ and height~$1$, we rescale the horizontal variable.
Define
\begin{equation}\label{e:HepVep}
   H^{\epsilon} (x_1,x_2) = H\paren[\Big]{ \frac{x_1}{\epsilon},  x_2 } \,,
  \quad \text{ and } \quad
  v^{\epsilon} = \frac{A_\epsilon}{\epsilon} \grad^\perp H^{ \epsilon} = \frac{A_\epsilon }{\epsilon}
  \begin{pmatrix}
  \partial_2 H^{\epsilon}\\ - \partial_1 H^{\epsilon}
      \end{pmatrix}\,,
\end{equation}
and let $T_{\epsilon} = T^{v^{\epsilon}}$.
By uniqueness of solutions to~\eqref{e:T} we see that $T_{\epsilon}$ satisfies $T_{\epsilon}( x_1 + 2 \epsilon, x_2) = T_{\epsilon}(x_1, x_2)$.
Thus, it is natural to make the change of variables
\begin{equation}\label{e:cvy}
  y_1 = \frac{x_1}{\epsilon}\,,
  \quad
  y_2 = x_2\,,
  \qquad\text{and}\qquad
  v = (v_1,v_2) = \grad^\perp_y H\,.
\end{equation}
In these coordinates we see that $T_{k,\epsilon}$ satisfies
\begin{equation}\label{e:Teps}
  A_\epsilon v \cdot \grad_y T_{\epsilon}
  - \frac{1}{2}\partial_{y_1}^2 T_{\epsilon}
  -  \frac{1}{2}\epsilon^2\partial_{y_2}^2 T_{\epsilon} = \epsilon^2 \,.
\end{equation}

Examining~\eqref{e:Teps} we see that in the horizontal direction the diffusion has strength~$1$.
However, since we impose periodic boundary conditions in this direction, there are no boundaries that provide a cooling effect directly felt by the horizontal diffusion.
In the vertical direction, the diffusion coefficient is~$\epsilon^2$, and so the cooling effect from the Dirichlet boundary~$\partial S_2$ will be felt in the domain in time $O(1/\epsilon^2)$.
Since our source (the right hand side of~\eqref{e:Teps}) is also $\epsilon^2$, we expect that the diffusion alone will ensure $T_\epsilon$ is of size $O(1)$ as $\epsilon \to 0$.
This would lead to the bound~$\mathcal E^{1,p}_q(\curE) \leq C$, which is far from optimal.

We claim that the convection term reduces this bound dramatically.
Indeed, through convection one can travel an~$O(1)$ distance in the vertical direction in time $1/A_\epsilon$.
Due to our no flow requirement $v \cdot \hat n = 0$ on $\partial S_2$, one can never reach the boundary of~$S_2$ through convection alone.
Thus, the cooling effect of the boundary~$\partial S_2$ must propagate into the domain through a combination of the effects of the slow vertical diffusion $\epsilon^2 \partial_{y_2}^2$ and the fast convection $A_\epsilon v \cdot \grad_y$.
Our aim is to estimate how much improvement this can provide over the crude $O(1)$ bound that can be obtained through diffusion alone.
This is our next result.
\begin{proposition}\label{p:TepBound}
  There exists a smooth Hamiltonian $H$ satisfying~\eqref{e:Ham1}--\eqref{e:Ham3}, and a constant $C$ such that the following holds.
  For every $\nu >0$,  and $A_\epsilon$  chosen such that $A_\epsilon \geq 1/\epsilon^\nu$ we have,
  \begin{equation}\label{e:TepBound}
    \norm{T_{\epsilon}}_{L^\infty}  \leq
    C\epsilon^2\paren[\Big]{ 1 + \frac{  \abs{\ln \epsilon}^{13}}{\epsilon\sqrt{A_\epsilon}} }
  \end{equation}
  for all sufficiently small~$\epsilon$.
  Here $T_\epsilon = T^{v^\epsilon}$, and~$v^\epsilon$ is given by~\eqref{e:HepVep}.
\end{proposition}
\begin{remark}
  We believe the bound~\eqref{e:TepBound} is true for every smooth, non-degenerate cellular flow $v$ (with a constant $C = C(v)$), provided $\nu \geq 2$.
  To obtain~\eqref{e:TepBound} for all $\nu > 0$, our proof requires the velocity field~$v$ to be exactly linear near the vertical cell boundaries.
  We do not know whether~\eqref{e:TepBound} remains true for $\nu \in (0, 2)$ without this assumption.
  We note, however, that choosing $\nu \in (0, 2)$ does not lead to an improved bound as in this range the constant term on the right of~\eqref{e:TepBound} will eliminate any benefit obtained from further increasing the amplitude.
\end{remark}
\begin{remark}
  For simplicity, the velocity field we construct to prove Proposition~\ref{p:TepBound} will be chosen to be exactly linear near cell corners.
  This assumption is mainly present as it leads to a technical simplification of the proof of Proposition~\ref{p:TepBound}.
  Since the proof of Proposition~\ref{p:enstrophy} only requires us to produce one velocity field~$v$ satisfying~\eqref{e:intTv}, we only state and prove Proposition~\ref{p:TepBound} for a specific cellular flow, instead of generic cellular flows.
\end{remark}

We prove Proposition~\ref{p:TepBound} using probabilistic techniques in the next section.
Proposition~\ref{p:enstrophy} follows immediately from Proposition~\ref{p:TepBound} by scaling.
\begin{proof}[Proof of Proposition~\ref{p:enstrophy}]
  By definition, we have
  \begin{equation*}
    v^{\epsilon}(x_1, x_2)
    = \frac{A_\epsilon}{\epsilon} \grad^\perp H^{ \epsilon}(x_1, x_2)
    = 
    \frac{A_\epsilon}{\epsilon^2} 
    \begin{pmatrix}
      \epsilon v_1\paren{y_1, y_2} \\
      v_2\paren{y_1, y_2}
    \end{pmatrix} \,,
  \end{equation*}
  and hence
  \begin{equation*}
    \grad_x v^\epsilon
      = \frac{A_\epsilon}{\epsilon^3}
  \begin{pmatrix}
    \epsilon \partial_{y_1} v_1 & \epsilon^2 \partial_{y_2} v_1 \\
    \partial_{y_1} v_2 & \epsilon \partial_{y_2} v_2
  \end{pmatrix}
  \end{equation*}
  Therefore, as $\epsilon \to 0$, we have
  \begin{align*}
    \mathscr E = \norm{v^\epsilon}_{W^{1,p}}
    = O\paren[\Big]{ \frac{A_\epsilon}{\epsilon^{3}}  } \,.
  \end{align*}
  Choosing $A_\epsilon = 1/\epsilon^\nu$,
  we have for large enough $\curE$,
  \begin{equation}\label{e:eEpsilon}
    \curE  = O\paren[\Big]{ \frac{1}{\epsilon^{3+\nu}}  } 
    \quad \text{and} \quad 
    \epsilon = O\paren[\Big]{ \frac{1}{\curE^{1/(3+\nu)}}}
  \end{equation}
  Combining this with~\eqref{e:TepBound}, we have  
  \begin{equation*}
    \norm{T_{\epsilon}}_{L^\infty} 
    \leq  C \paren[\Big]{\epsilon^2 +  \epsilon^{1 + \nu/2} \abs{\ln{\epsilon}}^{13} } \,.
  \end{equation*}
  Rewriting this in terms of~$\mathscr E$ using~\eqref{e:eEpsilon} and choosing $\nu = 2$ shows
  \begin{equation*}
    \norm{T_{\epsilon}}_{L^\infty} 
     \leq  
     C \frac{\abs{\ln \curE}^{13}}{\curE^{2/5}}  \,.
  \end{equation*}
  This implies~\eqref{e:E1upperBd} as desired.
\end{proof}

\section{Exit time from tall and thin cells (proof of Proposition~\ref{p:TepBound})}\label{s:TepBound}

Our aim in this section is to prove Proposition~\ref{p:TepBound}.
For ease of notation we will now write $v = v^\epsilon$, $T = T_\epsilon$, $A = A_\epsilon$.
Let $Z^\epsilon$ be a solution to the SDE
\begin{equation}\label{eq:degenerateProcess}
  dZ^\epsilon_t = A v(Z^\epsilon) \, ds
  + \sigma
  \, dB_t \,,
  \qquad\text{where}\qquad
  \sigma \defeq  \begin{pmatrix} 1 & 0\\0 & \epsilon \end{pmatrix}\,.
\end{equation}
Here~$B$ is a standard two dimensional Brownian motion.
For convenience let $Z^\epsilon = (Z^{\epsilon}_1, Z^{\epsilon}_2)$, and let
\begin{equation}\label{e:tauDef}
  \tau^\epsilon = \inf \set{ t \st Z^{\epsilon}_{2,t} \not\in (-1, 1) }
\end{equation}
be the first exit time of $Z^\epsilon$ from the strip $S_2$.
(Here the notation~$Z^\epsilon_{2,t}$ refers to~$(Z^\epsilon_2)_t$, the value of the process~$Z^\epsilon_2$ at time~$t$.)
By the Dynkin formula we know~$T_\epsilon(z) = \epsilon^2 \E^z \tau^\epsilon$.

Before delving into the details of the proof of Proposition~\ref{p:TepBound}, we now briefly explain the main idea.
Consider many tracer particles evolving according to~\eqref{eq:degenerateProcess}.
First, we note that particles near $\partial S_2$ get convected away from~$\partial S_2$ in time $O(1/A)$.
In this time, these particles can travel a distance of $O(\epsilon / \sqrt{A})$ in the vertical direction through diffusion.
Thus, if we can ensure particles get to within a distance of $O(\epsilon / \sqrt{A})$ from $\partial S_2$, then they will exit quickly with probability at least $p_0$, for some small $p_0 > 0$ that is independent of $\epsilon$.

We claim that in the boundary layer, every $O(1/\sqrt{A})$ seconds%
\footnote{
  The diffusion may carry particles into the interior of the cell before they exit at $\partial S_2$.
  These particles will now take~$O(1/\sqrt{A})$ time to return to the boundary layer, which is why the time taken here is~$O(1/\sqrt{A})$, and not the convection time $O(1/A)$.
}
tracer particles will pass within a distance of $O(\epsilon / \sqrt{A})$ from $\partial S$.
Every pass has an $O(\epsilon)$ probability of being within $\epsilon/\sqrt{A}$ away from $\partial S_2$, and so a probability $O(\epsilon)$ of exiting from $\partial S_2$.
This suggests
\begin{equation}\label{e:tauIdealBd}
  \sup_{z \in S_2} \E^z \tau^\epsilon \leq C \paren[\Big]{
    1
    + \frac{\epsilon}{\sqrt{A}}
    + \frac{ (1-\epsilon)  2\epsilon}{\sqrt{A}}
    + \frac{ (1-\epsilon)^2 3\epsilon}{\sqrt{A}} + \cdots
  }
  = C \paren[\Big]{1 + \frac{1}{\epsilon \sqrt{A}} }\,,
\end{equation}
which is dramatically better than the crude $O(1/\epsilon^2)$ bound obtained by using diffusion alone.

A second look at the above argument suggests that~\eqref{e:tauIdealBd} should have a logarithmic correction.
Indeed, the flow~$v$ has hyperbolic saddles at cell $\set{-1, 0, 1} \times \Z$ which causes a logarithmic slow down of particles close to it.
As a result, we are able to prove the following bound on $\E \tau^\epsilon$.
\begin{proposition}\label{p:tauBd}
  Let $\nu >0$ and $A \geq 1/\epsilon^\nu$.
  There exists a cellular flow~$v$ and a constant $C$ such that
  \begin{equation}\label{e:tauBd}
    \sup_{z \in S_2} \E^z \tau^\epsilon \leq C \paren[\Big]{1+  \frac{ |\ln \epsilon|^{13}}{\epsilon \sqrt{A}} } \,,
  \end{equation}
  holds for all sufficiently small~$\epsilon$.
\end{proposition}

Of course Proposition~\ref{p:tauBd} immediately implies Proposition~\ref{p:TepBound}.
\begin{proof}[Proof of Proposition~\ref{p:TepBound}]
  Since $T(z) = \epsilon^2 \E^z \tau^\epsilon$, the estimate~\eqref{e:tauBd} implies~\eqref{e:TepBound} as desired.
\end{proof}

We now describe the flow~$v$ that will be used in Proposition~\ref{p:tauBd}.
As remarked earlier, we expect Proposition~\ref{p:tauBd} to hold for any generic non-degenerate cellular flow.
However, the specific form we describe below simplifies many technicalities.
For notational convenience, we will now restrict our attention to the rectangle
\begin{equation}\label{e:OmegaPrime}
  \Omega' \defeq (0, 2) \times (-1, 1)\,.
\end{equation}

\begin{asparaenum}[\itshape{Assumption} 1:]
  \item\label{A1}
  The function~$H \colon \R^2 \to [-1, 1]$ is $C^2$ with $\norm{H}_{C^2} \leq 100$ and is $2$-periodic in both $x_1$ and $x_2$.
  The level set $\set{H = 0}$ is precisely  $(\R \times \Z) \cup (\Z \times \R)$.
  Moreover, $H(1/2, 1/2) = 1$, $H(3/2, 1/2) = -1$ and these both correspond to non-degenerate critical points of~$H$.
  All other critical points of~$H$ are hyperbolic and lie on the integer lattice $\Z^2$.

  \item\label{A2}
  There exists $c_0 \in(0,  1/10)$ such that
  for
  \begin{equation}\label{e:Q0def}
    Q_0 \defeq (-2c_0, 2c_0)^2
  \end{equation}
  we have
  \begin{equation}\label{e:Hquadratic}
    H(x_1,x_2) = \begin{cases}
      x_1 x_2        & (x_1, x_2) \in Q_0\,,           \\
      (1-x_1) x_2    & (x_1, x_2) \in Q_0 + (1, 0)\,,  \\
      x_1(1-x_2)     & (x_1, x_2) \in Q_0 + (0, 1) \,, \\
      (1-x_1)(1-x_2) & (x_1, x_2) \in Q_0 + (1,1)\,.
    \end{cases}
  \end{equation}

  \item\label{A3}
  There exists a constant $h_0$ such that for $x \in \set{\abs{H} < h_0}$ and $i \in \set{1,2}$,
  \begin{equation*}
    \sign \partial_i^2 H = - \sign H  \,.
  \end{equation*}

  \item \label{A4}
  In the region $\set{\abs{H} \leq h_0} \cap (i + (-c,c))\times \R  $, where $i \in \Z$, 
  \begin{equation}
    \partial_1 v_2 = - \partial^{2}_{1} H = 0 \,.
  \end{equation}
\end{asparaenum}

Apart from non-degeneracy and normalization, the main content of the first assumption is that $H$ only has one critical point in the interior of every square of side length~$1$ with vertices on the integer lattice.
This is the main geometric restriction imposed on the Hamiltonian~$H$.
Assumptions~\ref{A2}--\ref{A3} are not necessary, but lead to technical simplifications of the proof.
Finally, Assumption~\ref{A4} is only required for the exit time bounds we obtain (Lemma~\ref{l:expectedExitBoundary}, below) to be valid when $A \leq 1 / \epsilon^2$.
Notice that in the proof of Proposition~\ref{p:enstrophy} we only use $A \approx 1/\epsilon^2$, and so Assumption~\ref{A4} is not essential.
We elaborate on this in Remark~\ref{r:A5}, below.

Now we split the proof of Proposition~\ref{p:tauBd} into two steps: estimating the time taken to reach the boundary layer, and then estimating the time taken to exit from the boundary layer.
In time $1/A$, the process $Z^\epsilon$ will typically travel a distance of
\begin{equation*}
    \delta \defeq \frac{\epsilon}{\sqrt{A}} \,,
\end{equation*}
    in the vertical direction.
    Given $\alpha > 0$ define the boundary layer (see Figure~\ref{f:boundary-layer}) $\mathcal B_\alpha$ by
\begin{equation*}
  \cH_{\alpha} = \cH_\alpha^\epsilon \defeq \set[\Big]{ \abs{H} <  \frac{\alpha}{\sqrt{A}}}
  \,.
\end{equation*}
\begin{figure}[htb]
    \begin{center}
      \includegraphics[width=.6\linewidth]{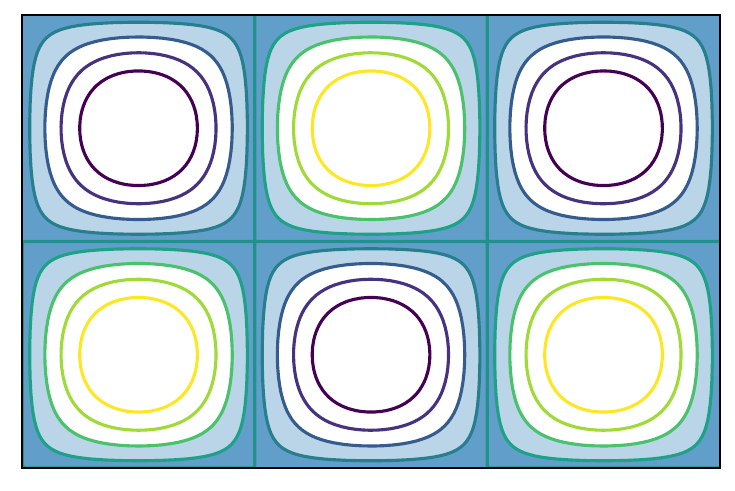}
      \caption{Boundary layer $\cB_1$ (dark blue) and boundary layer $\cB_5$ (union of light and dark blue). }
      \label{f:boundary-layer}
    \end{center}
  \end{figure}

\begin{lemma} \label{l:expectedExitBoundary}
  Let $\nu >0$ and
  suppose $A \geq 1/\epsilon^\nu$.
  There exists a constant $C$ such that
  \begin{equation}
    \sup_{z\in \bar{\cH}_1}\E^z  \tau^\epsilon
    \leq \frac{ C\abs{\ln\delta}^{13}  }{\epsilon \sqrt{A} } \,.
  \end{equation}
  Here $\bar{\mathcal B}_1$ denotes the closure of~$\mathcal B_1$.
\end{lemma}
\begin{remark}\label{r:A5}
  In the proof of Lemma~\ref{l:expectedExitBoundary} we will see that if $H$ doesn't satisfy Assumption~\ref{A4}, then Lemma~\ref{l:expectedExitBoundary} is only valid if $\nu \geq 2$ (see Remark~\ref{r:A51}, below).
  It turns out that choosing $\nu \leq 2$ provides no additional advantage in the proof of Proposition~\ref{p:tauBd}.
  This is because when $\nu \leq 2$, the constant term on the right of~\eqref{e:tauBd} dominates, and  we get no improvement on $\E^z \tau^\epsilon$.
\end{remark}

\begin{lemma}\label{lem:innerEst}
  For $\alpha > 0$ define
  \begin{equation}\label{e:etaDef}
     \eta_\alpha = \eta^\epsilon_\alpha \defeq
     \inf\set[\big]{ t>0 \st Z^\epsilon_t \in \partial \mathcal B_\alpha  }
  \end{equation}
  be the first time the process $Z^\epsilon_t$ hits $\partial \mathcal B_\alpha$.
  There exists a constant $C$, independent of~$\alpha$, such that
  \begin{equation*}
    \sup_{z\in \cH_\alpha^c} \E^z \eta^\epsilon_\alpha \leq C
  \end{equation*}
  for all sufficiently small~$\epsilon$.
  (Here $\mathcal B_\alpha^c$ is the complement of $\mathcal B_\alpha$.)
\end{lemma}
A proof of Lemma~\ref{lem:innerEst} using a blow-up argument can be found in~\cite{IshiiSouganidis12}.
We present a different proof of this fact (in Section~\ref{a:innerEst}, below) by constructing a supersolution based on the Freidlin averaging problem~\cite{FreidlinWentzell12}.

Momentarily postponing the proofs of Lemmas~\ref{l:expectedExitBoundary} and~\ref{lem:innerEst}, we prove Proposition~\ref{p:tauBd}.

\begin{proof}[Proof of Proposition~\ref{p:tauBd}]
  If $z \not\in \mathcal B_1$, the strong Markov property, Lemmas~\ref{l:expectedExitBoundary} and~\ref{lem:innerEst} imply
  \begin{align}
    \nonumber
    \E^z \tau^\epsilon  
    &=  \E^z \eta^\epsilon_1 + ( \tau^\epsilon - \eta^\epsilon_1 )
    =  \E^z \paren[\big]{ \eta^\epsilon_1 + ( \tau^\epsilon - \eta^\epsilon_1 ) \st \cF_{\eta^\epsilon_1}} \\
    \label{e:ezt1}
    &\leq C + \E^z \sup_{z' \in \bar B_1} \E^{z'} \tau^\epsilon
    \leq C\paren[\Big]{  1 + \frac{ \abs{\ln\delta}^{13}  }{\epsilon \sqrt{A} } } \,.
  \end{align}
  If $z\in \mathcal B_1$, then Lemma~\ref{l:expectedExitBoundary} directly implies~\eqref{e:ezt1}.
  Thus in either case we have~\eqref{e:tauBd}, as desired.
\end{proof}

\section{Exit from the Boundary layer (Lemma~\ref{l:expectedExitBoundary})}\label{s:bl}
In this section, we will prove Lemma~\ref{l:expectedExitBoundary}.
We will fix
 $\nu >0$ and suppose $A \geq 1/\epsilon^\nu$ as in the hypothesis of Lemma~\ref{l:expectedExitBoundary} through out this section.
 Furthermore, for notational convenience, we will now drop the explicit $\epsilon$ dependence from $Z^\epsilon$ and $A$.

 The main idea behind the proof of Lemma~\ref{l:expectedExitBoundary} is to focus our attention on trajectories in the boundary layer~$\mathcal B_1$, until they leave the bigger boundary layer~$\mathcal B_5$.
 Our first lemma estimates the chance of starting in~$\mathcal B_1$ and exiting the strip~$S_2$, before exiting the bigger boundary layer~$\mathcal B_5$.
\begin{lemma} \label{lem:lowerBoundExit}
  There exists a constant $C>0$, independent of $\epsilon$, such that
  \begin{equation} \label{ine:exitprobabilitynearbd}
    \inf_{z\in \cH_1}
    \P^z(  \tau^\epsilon <\eta^\epsilon_{5}) \geq \frac{C \epsilon}{\abs{\ln \delta}^{12}}
  \end{equation}
  for all sufficiently small~$\epsilon$.
\end{lemma}

Our next lemma estimates the amount of time the process takes to exit the bigger boundary layer~$\mathcal B_5$ (the union of the light and dark blue regions in Figure~\ref{f:boundary-layer}).
\begin{lemma}\label{lem:upperBoundEscape}
  There exists a constant $C$ such that
  \begin{equation} \label{ine:expectedescapetoinner}
    \sup_{z\in \cH_1} \E^z \eta^\epsilon_5 \leq \frac{C \abs{\ln \delta}}{A}
  \end{equation}
  for all sufficiently small~$\epsilon$.
\end{lemma}

Finally, we estimate the time taken for the process to return to the boundary layer~$\mathcal B_1$ starting from the boundary of the bigger boundary layer~$\mathcal B_5$.
Since trajectories may travel further inward this step is slower in comparison and takes $O(\abs{\ln \delta} / \sqrt{A})$.
\begin{lemma}\label{lem:upperBoundReturn}
  There exists a constant $C$ such that there exists an $\epsilon_0$, where
  \begin{equation} \label{ine:expectedbacktobdlayer}
    \sup_{z\in \partial \mathcal B_5}  \E^z \eta^\epsilon_1 \leq C \frac{\abs{\ln \delta}}{\sqrt{A}}
  \end{equation}
  for all $\epsilon <\epsilon_0$.
\end{lemma}

Momentarily postponing the proofs of Lemmas~\ref{lem:lowerBoundExit}--\ref{lem:upperBoundReturn}, we prove Lemma~\ref{l:expectedExitBoundary}.
\begin{proof}[Proof of Lemma~\ref{l:expectedExitBoundary}]
  In this proof, the constant $C$ may vary from line to line but does not depend on $\epsilon$.
  We first define two sequences of barrier stopping times,
  \begin{align*}
     & \sigma'_0 = 0 \,,                                                                                        &  & \tilde \sigma_0 = \inf\set[\big]{ t \geq \sigma'_0 \st Z^\epsilon_t \in \partial \mathcal B_5   } \,, \\
     & \sigma'_n = \inf\set[\big]{ t \geq \tilde \sigma_{n-1} \st Z^\epsilon_t \in \partial \mathcal B_1  } \,, &  & \tilde \sigma_n = \inf\set[\big]{ t\geq \sigma'_n \st Z^\epsilon_t \in \partial \mathcal B_5  } \,.
  \end{align*}
 We have 
  \begin{align}
    \nonumber
    \E^z \tau^\epsilon
     & = \int_0^\infty \P^z \paren[\big]{ \tau^\epsilon \geq t  } \, dt
    \\
    \nonumber
     & = \E^z\sum_{n=1}^\infty \int_{\sigma'_{n-1}}^{\sigma'_n} \one_{\set{\tau^\epsilon \geq t}} \, dt
    \leq \sum_{n=1}^\infty \E^z
    \one_{\set{\tau^\epsilon \geq \sigma'_{n-1}}}
    (\sigma'_n - \sigma'_{n-1})
    \\
    \nonumber
     & = \sum_{n=1}^\infty \E^z
    \one_{\set{\tau^\epsilon \geq \sigma'_{n-1}}}
    \E^{Z^\epsilon(\sigma'_{n-1})}
    \sigma'_1
    \\
    \label{e:Etau1}
     & \leq \sum_{n=1}^\infty \P^z
    \paren{\tau^\epsilon \geq \sigma'_{n-1}}
    \sup_{z' \in \partial \mathcal B_1} \E^{z'} \sigma'_1\,.
  \end{align}
  We will now estimate each term on the right.

  First, by the strong Markov property and Lemmas~\ref{lem:upperBoundEscape}--\ref{lem:upperBoundReturn} we have
  \begin{equation}\label{ine:boundsigma1}
    \E^z \sigma'_1
    = \E^z \paren[\big]{
      \tilde \sigma_0 + \E^{Z^\epsilon(\tilde \sigma_0)} \eta^\epsilon_1 }
    \leq \E^z \paren[\Big]{
      \eta^\epsilon_5 +
      \sup_{z' \in \partial \mathcal B_5} \E^{z'} \eta^\epsilon_1 }
    \leq \frac{C \abs{\ln \delta}}{\sqrt{A}} \,.
  \end{equation}
  for every~$z \in \partial \mathcal B_1$.
  To estimate $\P^z( \tau^\epsilon \geq \sigma'_n )$, we use Lemma~\ref{lem:lowerBoundExit} and the fact that $\sigma'_1 \geq \tilde \sigma_0 = \eta^\epsilon_5$ to obtain
  \begin{equation*}
    \sup_{z \in \partial \mathcal B_1} \P^z( \tau^\epsilon \geq \sigma'_1 )
    \leq \sup_{z \in \partial \mathcal B_1} \P^z( \tau^\epsilon \geq \eta^\epsilon_5 )
    = 1 - \inf_{z \in \partial \mathcal B_1} \P^z( \tau^\epsilon <\eta^\epsilon_5 )
    \leq 1 - \frac{C \epsilon}{\paren{\ln \delta}^{12}}\,.
  \end{equation*}
  Now, by the strong Markov property,
  \begin{align*}
    \sup_{z\in \cH_1}\P^z \paren[\big]{ \tau^\epsilon \geq \sigma'_n  }
     & = \sup_{z\in \cH_1}
    \E^z
    \paren[\big]{
      \one_{\set{\tau^\epsilon \geq \sigma'_{n-1}}}
      \E^{Z^\epsilon(\sigma'_{n-1})}
      \one_{\set{\tau^\epsilon \geq \sigma'_1 }}
    }
    \\
     & \leq \sup_{z\in \cH_1}
    \E^z
    \one_{\set{\tau^\epsilon \geq \sigma'_{n-1}}}
    \sup_{z' \in \partial \mathcal B_1}
    \P^{z'}
    \paren{\tau^\epsilon \geq \sigma'_1 }
    \\
     & \leq
    \paren[\Big]{ 1 - \frac{C \epsilon}{\paren{\ln \delta}^{12}} }
    \E^z
    \one_{\set{\tau^\epsilon \geq \sigma'_{n-1}}}\,.
  \end{align*}
  Hence by induction
  \begin{equation}\label{e:tauGeqSigmaN}
    \sup_{z\in \cH_1}\P^z \paren[\big]{ \tau^\epsilon \geq \sigma'_n  }
    \leq \paren[\Big]{ 1 - \frac{C \epsilon}{\abs{\ln \delta}^{12}} }^n\,,
  \end{equation}
  for all $n \in \N$.

  Using~\eqref{ine:boundsigma1} and~\eqref{e:tauGeqSigmaN} in~\eqref{e:Etau1} yields
  \begin{equation*}
    \E^z \tau^\epsilon
    \leq \frac{C \abs{\ln \delta}}{\sqrt{A}}
    \sum_{n=0}^\infty
    \paren[\Big]{1 - \frac{C \epsilon}{\abs{\ln \delta}^{12}} }^n
  \end{equation*}
  finishing the proof.
\end{proof}

\subsection{Proof of Lemma~\ref{lem:lowerBoundExit}}
In this subsection, we will give the proof of Lemma~\ref{lem:lowerBoundExit}.
Let the coordinate processes of~$Z$ be $Z_1$ and $Z_2$ respectively (i.e.\ $Z = (Z_1, Z_2)$).
Define $\gamma_t$ to be the deterministic curve satisfying the ODE
\begin{equation}\label{e:gamma}
  \partial_t \gamma_t  = A v(\gamma_t) \,.
\end{equation}
We again need a few results to prove Lemma~\ref{lem:lowerBoundExit}.

By symmetry and the reflection principle, when $Z$ wanders into the lower half of the domain $(0,2)\times (-1,0)$, its behavior is mirrored by $-Z$,
 which is again on the upper half of the domain $(0,2) \times (0,1)$. 
Hence, without loss of generality, we may restrict our attention to the upper half of the domain and all the lemmas below are stated in this context.

The first result we state is a ``tube lemma'' estimating the probability that the process $Z$ stays within a small tube around the deterministic trajectories.
This is well studied and many such estimates can be found in the literature (see for instance~\cite{FreidlinWentzell12}).
The standard estimates, however, work well for times of order $1/A$.
Due to the degeneracy, and the hyperbolic saddles near cell corners, we need an estimate that works for time scales of order $\abs{\ln \delta}  / A$.
We state this estimate here.

\begin{lemma}\label{l:tubecorner}
  Let $z_0\in (0,2)\times (0,1) \cap \paren[\big]{Q_0/2+ (j, k)}$ where $(j,k) \in \set{0,1,2}\times\set{0,1}$ and~$Q_0$ is as in~\eqref{e:Q0def}.
  Let $\gamma$ satisfy~\eqref{e:gamma} with $\gamma_0 = z_0$, and define
  \begin{equation}\label{e:Ttubecorner}
    T \defeq \inf\set{t>0 \st \abs{\gamma_{2,t}- 1} \leq \delta \text{ or } \abs{\gamma_{1,t}-1} = c_0 \text{ or } \abs{\gamma_{2,t}-1} = c_0}\,.
  \end{equation}
  Then there exists $\epsilon_0$ so that for every $\epsilon < \epsilon_0$,
  \begin{equation*}
    P^{z_0}\paren[\Big]{
      \sup_{0 \leq t \leq T}
  \abs{Z_{i,t} -  \gamma_{i,t}}  \leq  \frac{\sigma_{ii}}{\sqrt{\abs{\ln \delta}A}} \,, \forall i \in \set{1,2}}  \geq \frac{C}{\abs{\ln\delta}^2}  \,.
  \end{equation*}
  Here we recall that $\sigma_{11} = 1$ and $\sigma_{22} = \epsilon$ are the diagonal entries in the matrix~$\sigma$ in~\eqref{eq:degenerateProcess}.
\end{lemma}
\begin{remark} \label{Tcorner}
  By a direct calculation, we can check that $T \leq \abs{\ln\delta}/A$.
\end{remark}
The proof of Lemma~\ref{l:tubecorner} uses the Girsanov theorem and is greatly simplified by the fact that $H$ is exactly quadratic near cell corners.
Since it is similar to the standard proofs, we present it in Appendix~\ref{s:tubelemmas}.

Once Lemma~\ref{l:tubecorner} is established it quickly gives an estimate on the probability of getting within a distance of $O(1/\sqrt{A})$  away from cell boundaries.
\begin{lemma} \label{cor:resetBound}
  Let $z_0 \in  \cH_1 \cap (0,2)\times (0,1)$.
  There exist constants $C, M>0$ such that for small enough $\epsilon$,
  \begin{equation} \label{resetBound}
    \P^{z_0}\paren[\big]{ \lambda_0 < \eta^\epsilon_{4M}  }
    \geq \frac{C}{\abs{\ln \delta}^2} \,.
  \end{equation}
  Here, $\lambda_0 \defeq \inf\set[\big]{ t>0 \st Z_t \in \set{  \dist(z, \partial \Omega') \leq M/\sqrt{A}}} $.
\end{lemma}
\begin{proof}
  Note first that by Taylor expansion of $H$, for small $\epsilon$ there exists $M > 0$ such that $\dist(z_0, \partial \Omega') \leq  M/\sqrt{A}$ for all $z_0$ outside the corners $Q_0/2 + (j,k)$, 
  where $(j,k) \in \set{0,1,2} \times \set{0, 1}$.
  So now, we assume $z_0 \in Q_0/2 + (j,k)$ for some 
  $(j,k) \in \set{0,1,2} \times \set{ 0, 1}$.
  For brevity, we only present the proof when $z_0 \in Q_0/2$,  as the other cases are identical.

  If $\dist(z_0, \partial \Omega') \leq 1/\sqrt{A}$ we are done, so we now suppose $z_0 \in Q_0/2$ with $\dist(z_0, \partial \Omega') > 1/\sqrt{A}$.
  Let $\gamma$ be the deterministic trajectory defined by~\eqref{e:gamma} with $\gamma_0 = z_0$, and let $T$ be as in~\eqref{e:Ttubecorner}.
  Note that since $\dist(z_0, \partial \Omega') > 1/\sqrt{A}$ we can not have $\abs{\gamma_{2,T}-1} \leq \delta$.
  Thus, either $\abs{\gamma_{1,T} - 1} = c_0$ or $\abs{\gamma_{2,T} - 1} = c_0$.
  In either case there exists a constant $M$ such that $\abs{\gamma_{2,T}- 1} \leq M/\sqrt{A}$ or $\abs{\gamma_{1,T} - 1}\leq M/\sqrt{A}$, respectively.
  Now using Lemma~\ref{l:tubecorner} we obtain~\eqref{resetBound} as desired.
\end{proof}
\begin{remark}
  For notational convenience, we assume that $M=1$ for the rest of the paper.
\end{remark}

Another consequence of Lemma~\ref{l:tubecorner} is a lower bound on the probability of reaching $O(\delta)$ away from the top boundary before re-entering the cell interior.
\begin{lemma} \label{lem:cornerest}
  Let $Q^\delta_{\text{top}} = (1-2c_0,1+2c_0)\times (1-4\delta,1)$ be a box of height $4\delta$ at the top of the cell corner.
  Let $\lambda \defeq \inf \set{t\geq 0 \st Z_t \in Q^\delta_{\text{top}}}$. Then, there exists a constant $C>0$ such that
  \begin{equation} \label{ine:estimate1}
    \inf_{z_0\in (1-\delta,1+\delta) \times (1-c_0,1)}
    \P^{z_0} \paren[\big]{  \lambda < \eta^\epsilon_4  }  \geq \frac{C}{(\ln \delta)^2} \,.
  \end{equation}
\end{lemma}
\begin{proof}
  Let $T = \inf\set[\big]{t>0 \st \abs{\gamma_{2,t} -1} \leq \delta}$ the time the deterministic process hits the top boundary layer with width $\delta$.
  By Lemma~\ref{l:tubecorner},
  there exists a constant $C>0$ so that
  \begin{equation*}
    P^{z_0}\paren[\Big]{
    \sup_{0 \leq t \leq T} \abs{Z_{i,t} -  \gamma_{i,t}}  \leq  \frac{\sigma_{ii}}{\sqrt{\abs{\ln \delta}A}} \,, \forall i \in \set{1,2} }
    \geq \frac{C}{(\ln\delta)^2}  \,.
  \end{equation*}
  As $z_0 \in (1-\delta,1+\delta) \times (1-c_0,1)$, $\gamma_{1,T} \in (1-c_0, 1+c_0)$. Therefore,
  \begin{equation*}
    \set[\Big]{
      \sup_{0 \leq t \leq T}
  \abs{Z_{i,t} -  \gamma_{i,t}}  \leq  \frac{\sigma_{ii}}{\sqrt{\abs{\ln \delta}A}} \,, \forall i \in \set{1,2} }
    \subseteq \set[\big]{ \eta^\epsilon_4 > \lambda  } \,,
  \end{equation*}
  from which~\eqref{ine:estimate1} follows.
\end{proof}

Next, we bound the probability of exiting from the top when trajectories start in~$Q^\delta_{\text{top}}$.

\begin{lemma} \label{lem:topest}
  There exists a constant $p_0>0$ such that
  \begin{equation} \label{ine:topest}
    \inf_{z_0 \in Q^\delta_{\text{top}}}\P^{z_0} \paren[\big]{ \tau^\epsilon < \eta^\epsilon_4  } \geq p_0 \,.
  \end{equation}
\end{lemma}
\begin{proof}
  Let $\tilde T = 1/A$.
  When $A$ is sufficiently large, we note that given $X_0 = z_0 \in Q^\delta_{\text{top}}$, there exists $n \geq 1$, independent of $\epsilon$, such that
  the deterministic flow $\gamma_t$ starting at $z_0$ still remains in the top edge of the boundary layer $\set{\abs{H} \leq n\delta} \cap (0,2)\times (1-n \delta, 1)$ for time $\tilde T$.
  Define $\tilde \gamma_t$ by
  \begin{equation*}
    \partial_t\tilde \gamma_t = A u(\tilde \gamma_t) \,,
  \end{equation*}
  where $u=(u_1,u_2)$ is chosen to satisfy the following condition
  $\tilde \gamma_t  = (\gamma_{1,t}, \tilde \gamma_{2,t})$, where $\gamma_{1,t}$ is the first coordinate of $\gamma$, and $\tilde \gamma_{2,t}$ is some continuous function such that
  \begin{equation*}
    \tilde \gamma_{2,0} = \gamma_{2,0}  \,,  \quad
    \abs{v_2 - u_2} \leq 2n\delta
    \quad \text{ and } \quad
    \tilde \gamma_{2,\tilde T} \geq n\delta \,.
  \end{equation*}
  An example of such $\tilde\gamma$ is $\tilde\gamma_t = ( \gamma_{1,t} , \gamma_{2,t} + 2An\delta t)$.
  By continuity of $Z$, we have
  \begin{equation*}
    E_3 \defeq \set[\Big]{ \sup_{0\leq t \leq \tilde T} \abs{ Z_{2,t} - \tilde \gamma_{2,t}} \leq \delta  }
    \subset
    \set[\big]{  \tau^\epsilon < \eta^\epsilon_4} \,.
  \end{equation*}
  Now a standard large deviation estimate will show that $\P^{z_0}(E_3) \geq p_\epsilon$, for some constant $C_\epsilon$ that vanishes as $\epsilon \to 0$.
  In order to prove Lemma~\ref{lem:topest}, we need to remove this $\epsilon$ dependence.
  We do this here using the fact that in this box $\abs{\partial_1 v_2} \leq O(\epsilon)$, and $\abs{v_2 - u_2} \leq O(\delta)$.
  We claim that if we go through the standard large deviation estimate with these additional assumptions, the constant $p_\epsilon$ can be made independent of~$\epsilon$.
  Since the details are not too different from the standard proof, we carry them out in Lemma~\ref{l:smallDev} in Appendix~\ref{s:tubelemmas}, below.
  Hence, we see that there exists a constant $p_0$ (independent of $z_0, \epsilon$) so that
  \begin{equation*}
    \P^{z_0}(E_3) \geq p_0 \,,
  \end{equation*}
  proving~\eqref{ine:topest}.
\end{proof}

\begin{lemma} \label{lem:sideest}
  Let  $\tilde \lambda \defeq \inf \set[\big]{ t\geq 0 \st Z_t \in (1-\delta ,1+\delta )\times (1-c_0,1)}$. There exists a constant $C>0$ such that
  \begin{equation} \label{ine:sidelowerbound}
    \inf_{z_0 \in \set{\dist\paren{z,\partial\Omega'} \leq 1/\sqrt{A}}}\P^{z_0}\paren[\big]{ \tilde\lambda < \eta^\epsilon_4 } \geq  \frac{C \epsilon}{\paren{\ln \delta}^{8}} \,.
  \end{equation}
\end{lemma}
\begin{proof}
  We give the proof where $z_0 \in \set{\dist\paren{z,\partial\Omega'} \leq 1/\sqrt{A}} \cap (0,1)\times (0,1)$. 
  The analysis is similar for $z_0 \in \set{\dist\paren{z,\partial\Omega'} \leq 1/\sqrt{A}} \cap (1,2)\times (0,1)$.
  Define the regions~$\Box_1$, \dots, $\Box_5$ by
  \begin{gather*}
    \Box_1 \defeq
    \paren[\Big]{1-\frac{1}{\sqrt{A}}, 1+\frac{1}{\sqrt{A}}} \times \paren[\Big]{\frac{1}{\sqrt{A}}, 1- \frac{1}{\sqrt{A}} } \,, \\
    \Box_2 \defeq
    \paren[\Big]{ \frac{1}{\sqrt{A}}, 1} \times \paren[\Big]{0,  \frac{1}{\sqrt{A}}}  \,, \\
    \Box_3 \defeq
    \paren[\Big]{0 ,  \frac{1}{\sqrt{A}}} \times \paren[\Big]{0, 1-\frac{1}{\sqrt{A}}}  \,, \\
    \Box_4 \defeq
    \paren[\Big]{0, 1-\frac{1}{\sqrt{A}} }\times \paren[\Big]{1-\frac{1}{\sqrt{A}},1 }  \,, \\
    \Box_5 \defeq
    \paren[\Big]{1-\frac{1}{\sqrt{A}},0 }^2 \,,
  \end{gather*}
  as shown in Figure~\ref{f:boxes}.
  \begin{figure}[htb]
    \begin{center}
      \includegraphics[width=.6\linewidth]{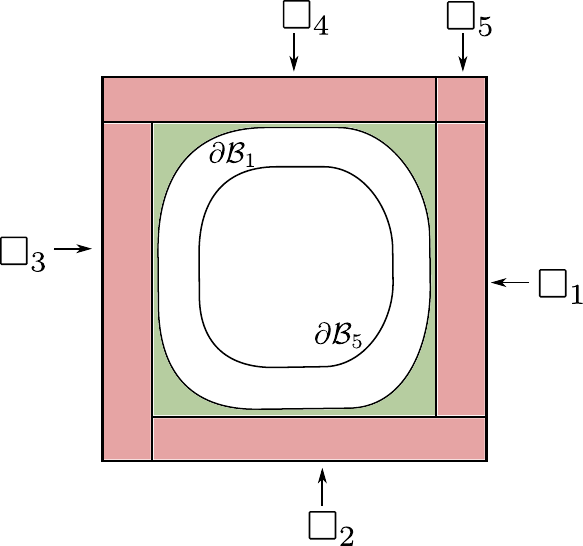}
      \caption{$\partial \cB_n$ and $\Box_i$.}
      \label{f:boxes}
    \end{center}
  \end{figure}
  If $\dist\paren{z_0,\partial\Omega'} \leq 1/\sqrt{A}$, then $z_0$ must be in one of the boxes~$\Box_1$, \dots, $\Box_5$.
  Suppose first $z_0 \in \Box_1$.
  Let $ \gamma(t)$ is the deterministic trajectory such that $\gamma_0 = z_0$,
  $T_0 \defeq \inf\set[\big]{t>0: \gamma_{2,t} = 1-c_0/2 }\leq m/A$ for some $m\geq 1$,
   and 
  \begin{equation*}
    E_4 \defeq
    \set[\Big]{  \sup_{0\leq t\leq T_0}\abs{Z_{1,t} -  \gamma_{1,t}} \leq \frac{2}{\sqrt{A}}   \,, 
    \sup_{0\leq t\leq T_0}\abs{Z_{2,t} - \gamma_{2,t}} \leq \frac{\epsilon}{\sqrt{A}} \,,
    \abs{Z_{1,T_0}} \leq \frac{\epsilon}{2\sqrt{A}} \,,
    } \,.
  \end{equation*}
  By continuity, we have that
  \begin{equation*}
    E_4
    \subset \set[\big]{ \tilde \lambda < \eta^\epsilon_4  } \,.
  \end{equation*}
  We claim
  \begin{equation} \label{ine:lowerboundary1}
    \P^{z_0}\paren[\big]{ \tilde \lambda < \eta^\epsilon_4  }
    \geq  \P^{z_0}(E_4)
    \geq C \epsilon\,,
  \end{equation}
  where $C>0$ independent of $z_0$.
  The proof of~\eqref{ine:lowerboundary1} is presented with the other tube lemmas we use in Appendix~\ref{s:tubelemmas}.
  We in fact prove a more general estimate (Lemma~\ref{l:STPC2} applied to the deterministic flow), from which~\eqref{ine:lowerboundary1} follows.

  Now, let $z_0 \in \Box_2$, define $\Box_{2R} = \Box_2 \cap [1-c_0,1]\times [0, 2/\sqrt{A}]$, and let $\lambda_1 = \inf\set[\big]{ t> 0 \st Z_t \in \Box_{2R}  }$.
  Proceeding as the case for $\Box_1$ with $\gamma(t)$ being the deterministic trajectory 
  so that $\gamma(0) = z_0$, $T_1 = \inf\set{t>0 \st \gamma_{1,t} =  c_0/2}$,
  we have
  \begin{equation} \label{ine:bottowBound1}
    \P^{z_0} \paren[\big]{ \lambda_1 < \eta^\epsilon_4  }
    \geq 
    \P^{z_0} \paren[\Big]{ \sup_{0\leq t\leq T_1}\abs{Z_{t} - \gamma_{t}}\leq \frac{1}{\sqrt{A}}  }
    \geq C  \,.
  \end{equation}
  To see why the last lower bound is true, we consider by It\^o formula,
  \begin{equation*}
    \sup_{0\leq t \leq T_1}\E^{z_0} \abs{Z_t - \gamma_t}^2 
    \leq 2A \norm{v}_{C^1} \int_0^{T_1}\E^{z_0} \sup_{0\leq t \leq T_1} \abs{Z_t - \gamma_t}^2
    + (\epsilon^2 + 1)T_1,
  \end{equation*}
  which, by Gronwall's inequality and Assumption~\ref{A1}, implies 
  \begin{equation*}
    \sup_{0\leq t \leq T_1}\E^{z_0} \abs{Z_t - \gamma_t}^2 
    \leq (1+ \epsilon^2)T_1 e^{ 200 T_1}\,.
  \end{equation*}
  Inequality~\eqref{ine:bottowBound1} follows by Chebychev's inequality.

  Now let $\lambda' =\inf\set[\big]{t\geq 0 \st Z_t \in  \Box_1}$.
  Using Lemmas~\ref{l:tubecorner}   and Markov property,  there exists a constant $C$ (independent of $z_0$) so that
  \begin{equation} \label{ine:lowerboundary21}
    \P^{z_0} \paren[\big]{ \lambda' < \eta^\epsilon_4  }
    \geq \P^{z_0} \paren[\big]{ \lambda_1 <\eta^\epsilon_4} \inf_{z_1 \in \Box_{2R}} \P^{z_1}\paren[\big]{ \lambda'<\eta^\epsilon_4  }
    \geq \frac{C}{(\ln\delta)^2} \,.
  \end{equation}
  Combining~\eqref{ine:lowerboundary1}, \eqref{ine:lowerboundary21} and using the Markov property gives
  \begin{equation*}
    \P^{z_0}\paren[\big]{ \tilde \lambda < \eta^\epsilon_4  }
    \geq \P^{z_0}\paren[\big]{\lambda' <\eta^\epsilon_4}
    \inf_{z_1 \in \, \Box_1 }\P^{z_1} \paren[\big]{  \tilde \lambda < \eta^\epsilon_4}
    \geq \frac{C\epsilon}{(\ln\delta)^2} \,.
  \end{equation*}
  Repeating this argument again for $\Box_3$, \dots, $\Box_5$ we see that we obtain an extra $C / \abs{\ln \delta}^2$ factor every time we pass a corner.
  Combining these estimates gives~\eqref{ine:sidelowerbound} as claimed.
\end{proof}

We are now ready to give the proof for Lemma~\ref{lem:lowerBoundExit}.

\begin{proof}[Proof of Lemma~\ref{lem:lowerBoundExit}]
  Let $z_0 \in \cH_1$ and
  denote $D_1 \defeq\set[\big]{\dist\paren{z,\partial\Omega'} \leq 1/\sqrt{A}}$,
  $D_2 \defeq (1-\delta, 1+\delta)\times (1-c_0,1)$ and $D_3 \defeq (1-2c_0,1+2c_0) \times (1-4\delta,1)$.
  As $\eta^\epsilon_4 < \eta^\epsilon_5$ when $z_0 \in \cH_1$, by Lemmas~\ref{cor:resetBound}--\ref{lem:sideest} and Markov property, we have that
  \begin{align*}
    \MoveEqLeft
    \P^{z_0} \paren{ \tau^\epsilon < \eta^\epsilon_5  }
    \geq \E^{z_0} \one_{\set{  \tau^\epsilon < \eta^\epsilon_5  }}
    \one_{\set{ \lambda < \eta^\epsilon_5  }}
    \one_{\set{  \lambda_0 < \eta^\epsilon_5}}
    \one_{\set{ \tilde \lambda < \eta^\epsilon_5}}
    \\
     & = \E^{z_0} \one_{\set{  \lambda_0 < \eta^\epsilon_5}}
    \E^{z_0} \paren[\Big]{  \one_{\set{  \tau^\epsilon < \eta^\epsilon_5  }}
      \one_{\set{  \tilde \lambda < \eta^\epsilon_5}}
      \one_{\set{ \lambda < \eta^\epsilon_5  }}
    \given \cF_{\lambda_0}  }                                  \\
     & = \E^{z_0} \one_{\set{  \lambda_0 < \eta^\epsilon_5}}
    \E^{Z_{\lambda_0}} \paren[\Big]{  \one_{\set{  \tau^\epsilon < \eta^\epsilon_5  }}
      \one_{\set{  \tilde \lambda < \eta^\epsilon_5}}
      \one_{\set{ \lambda < \eta^\epsilon_5  }}
    }                                                          \\
     & \geq \E^{z_0} \one_{\set{ \lambda_0 < \eta^\epsilon_5}}
    \inf_{z_1 \in D_1} \E^{z_1} \paren[\Big]{
      \one_{\set{ \lambda < \eta^\epsilon_5  }}
      \one_{\set{  \tilde \lambda < \eta^\epsilon_5}}
      \one_{\set{  \tau^\epsilon < \eta^\epsilon_5  }}
    }                                                          \\
     & \geq \E^{z_0} \one_{\set{ \lambda_0 < \eta^\epsilon_5}}
    \inf_{z_1 \in D_1}
    \E^{z_1} \one_{\set{ \tilde \lambda < \eta^\epsilon_5  }}
    \inf_{z_2 \in D_2}
    \E^{z_2}
    \one_{\set{  \lambda < \eta^\epsilon_5}}
    \inf_{z_3 \in D_3} \E^{z_3}\one_{\set{\tau^\epsilon < \eta^\epsilon_5}}
    \\
     & \geq  \frac{C \epsilon}{\abs{\ln \delta}^{12} }\,,
  \end{align*}
  where $C$ is independent of $z_0$. Taking the infimum over $z_0$, we achieve the desired result.
\end{proof}

\subsection{Proof of Lemma~\ref{lem:upperBoundEscape}}
In this subsection, we give a proof of Lemma~\ref{lem:upperBoundEscape}.
The strategy then will be similar to that of the proof of Lemma~\ref{lem:lowerBoundExit} as will will estimate the probability for a typical particle to successfully enter the inner region after each time it goes around the boundary layer $\cH_5$.
To do this, we first need a few results.
\begin{lemma}\label{lem:escapeInner}
  Let $\tilde \Box_1 = \cH_5 \cap \set{ x_2 \in [c_0, 1-c_0]}$.
  There exists a constant $p_0$ such that
  \begin{equation} \label{in:escapeInner}
    \inf_{z_0 \in \tilde \Box_1}\P^{z_0}\paren[\Big]{ \eta^\epsilon_5 < \frac{1}{A}  } \geq p_0 \,.
  \end{equation}
\end{lemma}
\begin{proof}
  Since we restrict our attention to  region of the boundary layer on the sides, for each $\epsilon > 0$ there exists an interval 
  $R_\epsilon$ with length $\abs{R_\epsilon} = 1/\sqrt{A}$ such that
  $$\dist\paren[\big]{R_\epsilon\times [c_0,1-c_0]  \,, 
  \cH_5 \cap \set{x_2 \in [c_0,1-c_0]}} = \frac{1}{\sqrt{A}}\,.$$
  Let $M$ be independent of $\epsilon$ such that 
  $$ R_\epsilon \times [c_0,1-c_0] \cup \paren[\big]{\cH_5 \cap  \set{x_2 \in [c_0,1-c_0]}} \subseteq \paren[\Big]{ 1-\frac{M}{\sqrt{A}}\,, 1+\frac{M}{\sqrt{A}}} \times [c_0,1-c_0] \,,$$ 
  and 
   $z_0 \in \tilde \Box_1$. 
  By Lemma~\ref{l:STPC2} 
  applied to the deterministic curve $\gamma$ (given by~\eqref{e:gamma}) with $\gamma_0 = z_0$, we have
  \begin{multline*}
    \P^{z_0}\paren[\Big]{ \eta^\epsilon_5 < \frac{1}{A}  }
    \\
    \geq  \P^{z_0} \paren[\Big]{ \sup_{0\leq t\leq 1/A} \abs{ Z_{1,t} - \gamma_{1,t}} \leq \frac{M}{\sqrt{A}} \,, 
    \sup_{0\leq t \leq 1/A} \abs{ Z_{2,t} - \gamma_{2,t}  }\leq \frac{\epsilon}{\sqrt{A}}\,, Z_{1,T_0} \in R_\epsilon  } \geq p_0 \,,
  \end{multline*}
  where $p_0$ is independent of $z_0$ as desired.
\end{proof}
\begin{lemma}\label{lem:lambda2est}
  Let $\tilde \lambda_2 = \inf\set[\big]{ t>0 \st Z_{2,t} \in \set{c_0, 1-c_0}   }$ and $z_0 \in \cB_{5}- \tilde \Box_1$.
  Then
  \begin{equation} \label{e:lambda2est}
    \lim_{\epsilon \to 0 } \inf_{\cH_5 - \tilde \Box_1}\P^{z_0}\paren[\Big]{
      \tilde \lambda_2 \leq \frac{5\abs{\ln\delta}}{A} } \geq 1 - \frac{C \ln A}{A^{1/4}} \,.
  \end{equation}
\end{lemma}
\begin{proof}
  Let $q \geq 2$ be some large number to be chosen later, and let $\tilde z_0$  be the closest point on $\set{ H = A^{-1/q}  }$ to $z_0$.
  Let $\tilde d = A \abs{z_0 - \tilde z_0}$
  and $\gamma_t$ be the deterministic curve (defined by~\eqref{e:gamma}) with~$\gamma_0 = \tilde z_0$.
  Note that, by Assumptions~\ref{A1}--\ref{A2},
  \begin{equation} \label{ine:projectionDistance}
      \frac{\tilde d}{A} \leq  \frac{C}{A^{1/2q}} \,.
  \end{equation}
  By It\^{o} formula, we have
  \begin{equation*}
      \E^{z_0}\abs{Z_t - \gamma_t}^2
  \leq \frac{\tilde d^2}{A^2}
    + 2 A \norm{v}_{C^1}
      \int_0^t \E^{z_0}\abs{Z_s - \gamma_s}^2 \, ds
      +(1 + \epsilon^2)t \,.
  \end{equation*}
  By Gronwall's inequality and Assumption~\ref{A1}, it follows that
  \begin{equation*}
    \begin{aligned}
      \E^{z_0}\abs{Z_t - \gamma_t}^2
      \leq \paren[\Big]{\frac{\tilde d^2}{A^2} + (1+\epsilon^2)t}e^{200 At}\,.
    \end{aligned}
  \end{equation*}
  Now, let $T_0 = \inf\set{ t >0: \gamma_{2,t} \in \paren{2c_0, 1-2c_0}}$, and note that $T_0 \leq D \ln A / (Aq)$ for some constant $D > 0$.
  By~\eqref{ine:projectionDistance}, we have
  \begin{equation*}
    \begin{aligned}
       & \P^{z_0}\paren[\Big]{ \abs{Z_{T_0} - \gamma_{T_0}} \geq \frac{c_0}{10}  }
      \leq
      \frac{100}{c_0^2}\paren[\Big]{\frac{C}{A^{2q}} + (1+\epsilon^2)\frac{D\ln A}{Aq}}e^{200D\ln A/q} \\
       & \leq
      C A^{200D/q-1}\ln A  \,.                    \\
    \end{aligned}
  \end{equation*}
  Picking $q$ such that $ 200D/q -1 < -1/2$, we have
  \begin{equation} \label{ine:lowerProjection}
    \P^{z_0}\paren[\Big]{ \abs{Z_{T_0} - \gamma_{T_0}} < \frac{c_0}{10}  }
    \geq
    1 - \frac{C \ln A}{A^{1/4}} 
    \,.
  \end{equation}
  As $q\geq 2$ , $T_0 < 5\abs{\ln\delta}/{A}$.
  Therefore,
  by continuity of $Z$, it follows that
  \begin{equation*}
    \set[\Big]{ Z_{2,T_0} \in [2 c_0, 1- 2 c_0] }  \subseteq
    \set[\Big]{ \tilde \lambda_2 \leq \frac{5 \abs{\ln\delta}}{A} }\,.
  \end{equation*}
  Combining this with~\eqref{ine:lowerProjection}, we deduce 
  \begin{equation*}
    \lim_{\epsilon\to 0}
    \inf_{\cH_5 - \tilde \Box_1}\P^{z_0}\paren[\Big]{ \tilde \lambda_2 \leq \frac{5 \abs{\ln\delta}}{A} } \geq 1 - \frac{C \ln A}{A^{1/4}} \, ,
  \end{equation*}
  as desired.
\end{proof}

We are now ready for the proof of Lemma~\ref{lem:upperBoundEscape}.
\begin{proof}[Proof of Lemma~\ref{lem:upperBoundEscape}]
  \restartsteps
  \step
  We first claim that for each $z_0 \in \cH_5$ and $\epsilon >0$,
  there exists a constant $C>0$, independent of $z_0$ and $\epsilon$, such that
  \begin{equation} \label{ine:lowerBoundBoundary}
    \P^{z_0} \paren[\Big]{ \sup_{0\leq t \leq 6\abs{\ln\delta}/A} \abs{H(Z_t)} > \frac{5}{\sqrt{A}} } \geq C \,.
  \end{equation}

  To prove this, suppose for contradiction there exists a sequence $\set{z_n, \epsilon_n}_{n=1}^\infty$ such that
  \begin{equation}\label{ine:lowerBoundBoundaryCont}
    \lim_{n\to\infty}
    \P^{z_n} \paren[\Big]{ \sup_{0\leq t \leq 6\abs{\ln\delta}/A} \abs{H(Z_t)} > \frac{5}{\sqrt{A}} } = 0 \,.
  \end{equation}
  Let $C_0$ be the lower bound in Lemma~\ref{lem:escapeInner} and
  denote $\tilde \lambda_1 = \inf\set[\big]{ t\geq 0 \st Z_t \in \tilde \Box_1 }$.
  By Lemma~\ref{lem:escapeInner} and the strong Markov property,
  \begin{equation*}
    \begin{aligned}
       & \P^{z_n}\paren[\Big]{ \sup_{0\leq t \leq 6\abs{\ln\delta}/A} \abs{H(Z_t)} > \frac{5}{\sqrt{A}} }                                                                                                  \\
       & \geq \E^{z_n}\paren[\Big]{
       \E^{z_n}\paren[\Big]{ \one_{\set[\big]{  \sup_{0\leq t \leq \tilde \lambda_1} \abs{H(Z_t)} \leq \frac{5}{\sqrt{A}}}} \one_{\set[\big]{ \tilde \lambda_1 \leq 5\abs{\ln \delta}/A  }}
      \one_{\set[\big]{\eta^\epsilon_{5} < \tilde \lambda_1+  1/A}}\st \cF_{\tilde \lambda_1} }              }                                                                                      \\
       & = \E^{z_n} \paren[\Big]{
      \one_{\set[\big]{   \sup_{0\leq t \leq \tilde \lambda_1} \abs{H(Z_t)} \leq \frac{5}{\sqrt{A}} } } \one_{\set[\big]{ \tilde \lambda_1 \leq 5\abs{\ln \delta}/A  }}
      \E^{Z_{\tilde \lambda_1}}
      \one_{\set[\big]{\eta^\epsilon_{5} < 1/A}}                             }                                                                                                                              \\
       & \geq \E^{z_n}\paren[\Big]{
      \one_{\set[\big]{ \sup_{0\leq t \leq \tilde \lambda_1} \abs{H(Z_t)} \leq \frac{5}{\sqrt{A}} }} \one_{\set[\big]{ \tilde \lambda_1 \leq 5\abs{\ln \delta}/A  }}}
      \inf_{z\in \tilde \Box_1}\E^{z}
      \one_{\set[\big]{\eta^\epsilon_{5} < 1/A}}                                                                                                                                                           \\
       & \geq C_0 \P^{z_n}\paren[\Big]{  \sup_{0\leq t \leq \tilde \lambda_1} \abs{H(Z_t)} \leq \frac{5}{\sqrt{A}} \,; \tilde \lambda_1 \leq \frac{5\abs{\ln \delta}}{A}  }  \,.
    \end{aligned}
  \end{equation*}
  The second equality follows from the fact that $\eta^\epsilon_5 > \tilde \lambda_1$ under the event 
  \begin{equation*}
  \set[\Big]{  \sup_{0\leq t \leq \tilde \lambda_1} \abs{H(Z_t)} \leq \frac{5}{\sqrt{A}}} \,.
  \end{equation*}

  We claim that for large enough $n$, we have
  \begin{equation*}
    \P^{z_n}\paren[\Big]{  \sup_{0\leq t \leq \tilde \lambda_1} \abs{H(Z_t)} \leq \frac{5}{\sqrt{A}} \,; \tilde \lambda_1 \leq \frac{5\abs{\ln \delta}}{A}  } \geq \frac{1}{2} \,,
  \end{equation*}
  which contradicts our assumption~\eqref{ine:lowerBoundBoundaryCont}.
  To see that this lower bound is true, we first note that  $z_i \not\in \tilde\Box_1$ by Lemma~\ref{lem:escapeInner}.
  Thus, we only consider the case $z_n \in \cH_5 - \tilde \Box_1$.

  Recall $\tilde \lambda_2 = \inf\set[\big]{ t>0 \st Z_{2,t} \in \set{c_0, 1-c_0}   }$.
  Observe that
  \begin{equation*}
    \begin{aligned}
       & \one_{\set{  \sup_{0\leq t \leq \tilde \lambda_1} \abs{H(Z_t)} \leq \frac{5}{\sqrt{A}}}} \one_{\set{ \tilde \lambda_1 \leq {5\abs{\ln \delta}}/{A} }} \\
       & =
      \one_{\set{ \sup_{0\leq t \leq \tilde \lambda_1} \abs{H(Z_t)} \leq \frac{5}{\sqrt{A}}}} \one_{\set{ \tilde \lambda_2 \leq {5\abs{\ln \delta}}/{A} }} \,.
    \end{aligned}
  \end{equation*}
  By~\eqref{e:lambda2est} and~\eqref{ine:lowerBoundBoundaryCont} and, we can pick $n$ large enough such that
  \begin{equation*}
    \begin{aligned}
       & \P^{z_n}\paren[\Big]{\sup_{0\leq t \leq \tilde \lambda_1} \abs{H(Z_t)} \leq \frac{5}{\sqrt{A}}
      \,;
      \tilde \lambda_1 \leq \frac{5\abs{\ln \delta}}{A} }                                                   \\
       & \geq
      \P^{z_n}\paren[\Big]{
      \sup_{0\leq t \leq 6\abs{\ln \delta}/A}  \abs{H(Z_t)} \leq \frac{5}{\sqrt{A}} \,;
      \tilde \lambda_2 \leq \frac{5\abs{\ln \delta}}{A} } \geq \frac{1}{2} \,.
    \end{aligned}
  \end{equation*}
  This is a contradiction, proving~\eqref{ine:lowerBoundBoundary} as desired.
  \step
  Once~\eqref{ine:lowerBoundBoundary} is established, we can estimate $\E \eta^\epsilon_5$ as the expected time to success of a Bernoulli trial using a similar argument as in the proof of Lemma~\ref{l:expectedExitBoundary}.
  Explicitly, let $\Delta t = 6\abs{\ln \delta}/A$, and observe that by~\eqref{ine:lowerBoundBoundary},
  \begin{equation*}
    \P^{z_0} \paren[\big]{ \eta^\epsilon_5 < \Delta t  }
    =
    \P^{z_0} \paren[\Big]{ \sup_{0\leq t \leq 6\abs{\ln\delta}/A} \abs{H(Z_t)} > \frac{5}{\sqrt{A}}} \geq C \,.
  \end{equation*}
  By the strong Markov property and estimate~\eqref{ine:lowerBoundBoundary},
  we have that for $i>1$,
  \begin{equation*}
    \begin{aligned}
       & \P^{z_0}\paren[\big]{ \eta^\epsilon_5 \geq i \Delta t  }
      = \E^{z_0} \E^{z_0} \paren[\big]{ \one_{\set{ \eta^\epsilon_5 \geq i \Delta t  }}
      \one_{\set{ \eta^\epsilon_5 \geq (i-1) \Delta t  }} \st \cF_{(i-1)\Delta t} }            \\
       & = \E^{z_0} \one_{\set{ \eta^\epsilon_5 \geq (i-1) \Delta t  }}
      \E^{Z_{(i-1)\Delta t}}  \one_{\set{ \eta^\epsilon_5 \geq  \Delta t  }}                   \\
       & \leq \E^{z_0} \one_{\set{ \eta^\epsilon_5 \geq (i-1) \Delta t  }}
      \sup_{z\in \cH_5 }\E^{z}  \one_{\set{ \eta^\epsilon_5 \geq  \Delta t  }}                 \\
       & = \E^{z_0} \one_{\set{ \eta^\epsilon_5 \geq (i-1) \Delta t  }}
      \paren[\big]{1- \inf_{z\in \cH_5 }\P^{z}  \paren[\big]{ \eta^\epsilon_5 <  \Delta t  } } \\
       & = \E^{z_0} \one_{\set{ \eta^\epsilon_5 \geq (i-1) \Delta t  }}
      \paren{1- C }
      \leq (1-C)^i \,,
    \end{aligned}
  \end{equation*}
  where $C$ is the constant in~\eqref{ine:lowerBoundBoundary}. Therefore,
  \begin{equation*}
    \begin{aligned}
      \E^{z_0} \eta^\epsilon_5 & = \int_0^\infty \P^{z_0}(\eta^\epsilon_5 \geq t) \, dt
      \leq  \sum_{i=1}^\infty\int_{(i-1)\Delta t}^{i\Delta t}
      \P^{z_0}\paren[\big]{\eta^\epsilon_5 \geq t} \, dt                                                                                                                                                    \\
                               & \leq \Delta t \sum_{i=0}^\infty \P^{z_0} \paren[\big]{\eta^\epsilon_5 \geq i\Delta t}  \leq \Delta t \sum_{i=0}^\infty (1- C)^i \leq \frac{ 6\abs{\ln\delta}}{(1-C)A}  \,,
    \end{aligned}
  \end{equation*}
  from which~\eqref{ine:expectedescapetoinner} follows immediately.
\end{proof}

\subsection{Proof of Lemma~\ref{lem:upperBoundReturn}}
In this subsection,
we restrict our attention to a particular cell $(0,1)\times(0,1)$ as the analysis is similar for $(1,2)\times (0,1)$.
Thus, assume for simplicity
that $\abs{H} = H$.
By Assumption~\ref{A3}, $\partial_i^2 H \leq 0$ for $i \in \set{1,2}$.
Let $z \in \overline \cB_1^c$ and denote $U_\epsilon(z) = \E^z \eta^\epsilon_1$.
Then, $U_\epsilon$ solves the following equation

\begin{equation} \label{eq:restrictedEq}
  \begin{dcases}
    - \partial_1^2 U_\epsilon - \epsilon^2 \partial_2^2 U_\epsilon + A v \cdot \grad U_\epsilon = 1 & \quad \text{ in } (0,1)^2 - {\cH}_1\,, \\
    U_\epsilon = 0                                                                                           & \quad \text{ on }  (0,1)^2 \cap \partial\cH_1\,.
  \end{dcases}
\end{equation}
In order to prove Lemma~\ref{lem:upperBoundReturn}, we construct an explicit supersolution  to~\eqref{eq:restrictedEq}, independent of $\epsilon$.
Recall by Lemma~\ref{lem:innerEst},
\begin{equation*}
  S \defeq \sup_{\epsilon >0} \norm{ U_\epsilon }_{L^\infty} < \infty \,.
\end{equation*}
Let $d_1 \ll 1$ be a small constant that will be chosen later, and define
\begin{gather*}
  \Lambda= \set[\Big]{ \frac{1}{\sqrt{A}} \leq |H| \leq d_1  }
  \\
  R_2 = \Lambda \cap \set{ y \in [c_0, 1-c_0]} \quad \text{ and } \quad
  R_1 = \Lambda - R_2 \,.
\end{gather*}

\begin{figure}[hbt]
  \begin{center}
    \includegraphics[width=.5\linewidth]{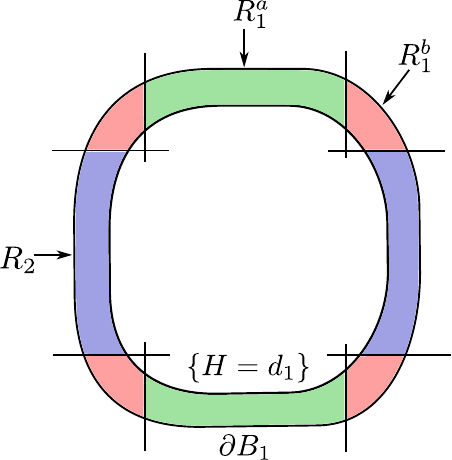}
    \caption{$\Lambda$, \textcolor{green!40!black}{$R_1^a$ (green)}, \textcolor{red}{$R_1^b$ (red)} and
      \textcolor{blue!70!black}{$R_2$ (blue)}.}
  \end{center}
\end{figure}

Denote by $(\theta, h)$ the curvilinear coordinate, where $\theta = \Theta(x_1, x_2)$ is the ``angle'' and $h= H(x_1,x_2)$ the level of the Hamiltonian $H$ (See Section~\ref{a:innerEst}).
Let
$f$ (to be specified later) be a smooth periodic function of $\Theta$  that satisfies
\begin{equation} \label{AssumptionOnf}
  \begin{gathered}
    0 < \inf f < \sup f < \infty \,, \\
    -\infty < \inf f'(\Theta) \leq \sup f'(\Theta) < -1 \quad \text{ on   } R_1 \,, \\
    \text{and} \quad \sup\abs{f''} <\infty  \,.
  \end{gathered}
\end{equation}

Then, consider the function
\begin{equation*}
  \phi =   \chi_1 + \chi_2 \,,
\end{equation*}
where
\begin{equation*}
  \chi_1 = - \frac{S}{d_1}  H \ln H \quad \text{ and } \quad
  \chi_2 =   - \frac{f(\Theta)}{A {H}}  + \frac{\norm{f}_{L^\infty}}{\sqrt{A}} \,.
\end{equation*}
By construction, $\phi(\Theta, H) \geq 0$ on $\Lambda$.
We claim that for an appropriate $f$, $\phi$ is a desired supersolution.

\begin{lemma} \label{lem:supersolution}
  Let $U_\epsilon$ be the solution to equation~\eqref{eq:restrictedEq}.
  Then, there exists a function $f$ that satisfies the requirement~\eqref{AssumptionOnf} so that for small enough $d_1$,
  \begin{equation*}
    \phi \geq U_\epsilon \quad \text{ on } \Lambda \,.
  \end{equation*}
\end{lemma}

Postponing the proof of this lemma, we now give the proof of Lemma~\ref{lem:upperBoundReturn}.

\begin{proof}[Proof of Lemma~\ref{lem:upperBoundReturn}]
  By construction,  on $\overline \cH_5 - \cH_1$ and for small enough $\epsilon$, we have
  $\frac{5}{\sqrt{A}} \leq d_1$. Therefore, when $H = 5/\sqrt{A}$,
  \begin{equation*}
    \phi \leq - \frac{S}{d_1}\frac{5}{\sqrt{A}}  {\ln \paren[\Big]{ \frac{5}{\sqrt{A}}  }}  + \frac{\norm{f}_{L^\infty}}{\sqrt{A}}
    \leq \frac{\abs{\ln\delta}}{\sqrt{A}}  \,.
  \end{equation*}
  It follows that
  \begin{equation*}
    \E^z\eta^\epsilon_1 =  U(z) \leq \phi(z) \leq \frac{\abs{\ln\delta}}{\sqrt{A}} \,,
  \end{equation*}
  for every $z \in \partial \cH_5$,
  as desired.
\end{proof}

\begin{proof}[Proof of Lemma~\ref{lem:supersolution}]
  \restartsteps
  \step
  Recall that $v = \grad^\perp H$ and $H \geq 1/\sqrt{A}$. We have that
  \begin{equation*}
    \grad \chi_2 = - \frac{f'(\Theta)}{A H} \grad \Theta +  \frac{f(\Theta)}{A H^2} \grad H \,,
  \end{equation*}
  \begin{align*}
    -\partial^2_1 \chi_2
     & =  \frac{1}{A} \paren[\Big]{ \frac{f''(\Theta)}{H} (\partial_1 \Theta)^2
    - 2\frac{f'(\Theta)}{H^2} \partial_1 \Theta \partial_1 H + \frac{f'(\Theta)}{H} \partial^2_1\Theta   }                                                                                                  \\
     & \quad  +\frac{1}{A} \paren[\Big]{  \frac{2 f(\Theta)}{H^3}  (\partial_1 H)^2 - \frac{f(\Theta)}{H^2} \partial^2_1 H }                                                                       \\
     & \geq \frac{1}{A } \paren[\Big]{ \frac{f''(\Theta)}{H} (\partial_1 \Theta)^2 - 2 \frac{f'(\Theta)}{H^2} \partial_1 \Theta \partial_1 H + \frac{f'(\Theta)}{H} \partial^2_1 \Theta      } \,,
  \end{align*}
  and
  \begin{equation*}
    -\partial_2^2 \chi_2 \geq
    \frac{1}{A } \paren[\Big]{ \frac{f''(\Theta)}{H} (\partial_2 \Theta)^2 - 2 \frac{f'(\Theta)}{H^2} \partial_2 \Theta \partial_2 H + \frac{f'(\Theta)}{H} \partial_2^2 \Theta      }  \,.
  \end{equation*}
  Therefore, by~\eqref{AssumptionOnf} and $H \geq 1/\sqrt{A}$,
  \begin{equation} \label{ine:chi2}
    - (\partial_1^2 + \epsilon \partial_2^2) \chi_2 \geq
    -\frac{2}{A} \paren[\Big]{ \frac{f'(\Theta)}{H^2}\paren[\big]{ \partial_1 \Theta \partial_1 H + \epsilon \partial_2 \Theta\partial_2 H } }
    - \frac{C}{\sqrt{A}}
    \,.
  \end{equation}

  \step
  On the other hand,
  \begin{equation*}
    \grad \chi_1 = -\frac{S}{d_1} (1 + \ln H) \grad H
  \end{equation*}
  and
  \begin{equation*}
    - \partial^2_1 \chi_1 =   \frac{S}{d_1} \partial^2_1 H (\ln H +1) +  \frac{S}{d_1}\frac{(\partial_1 H)^2}{H}
  \end{equation*}

  We  note that there exists a function $\rho = \rho(x) > 0$
  that
  \begin{equation*}
    \grad \Theta = \rho(x) \grad^\perp H = \rho(x) v(x)  \,,
  \end{equation*}
  and
  $  \lambda_1 \leq \rho \leq \lambda_2$ on  $\set[\big]{\abs{H} \leq c_0}$
  for some $0< \lambda_1 <\lambda_2$.
  Therefore, by~\eqref{ine:chi2} and $H\geq 1/\sqrt{A}$,
  \begin{align}
     & - \partial^2_1 \phi - \epsilon \partial_2^2 \phi + A v\cdot \grad \phi                                                                                                        \nonumber \\
     & \quad \geq \frac{S}{d_1} \partial^2_1 H (\ln H+1) + \frac{S}{d_1}\frac{(\partial_1 H)^2}{H} -  \frac{f'(\Theta) \abs{\grad H}^2 }{ H} \rho
    \label{ine:supersol}                                                                                                                                                                                \\
     & \qquad -\frac{2}{A} \paren[\Big]{ \frac{f'(\Theta)}{H^2}\paren[\big]{ \partial_1 \Theta \partial_1 H + \epsilon \partial_2 \Theta\partial_2 H } }
    - \frac{C}{\sqrt{A}}    \nonumber \,.
  \end{align}

  Recall
  \begin{equation*}
    R_2 = \Lambda \cap \set{ z_2 \in [ c_0, 1-c_0]} \quad \text{ and } \quad
    R_1 = \Lambda - R_2 \,.
  \end{equation*}
  We would like to estimate the above quantity in $R_1$ and $R_2$.

  \step
  For $R_1$, we decompose this set further
  \begin{equation*}
    R_1^a = R_1 \cap \set{  c_0 \leq {z_1} \leq 1-c_0} \quad \text{ and } \quad R_1^b = R_1 - R_1^a \,.
  \end{equation*}
  In $R_1^a$, there exists a constant $\tilde C$ such that $\abs{\grad H}^2 \geq \tilde C$. Therefore, by~\eqref{AssumptionOnf}, \eqref{ine:supersol} and $H\geq 1/\sqrt{A}$,
  \begin{align*}
    - \partial^2_1 \phi - \epsilon \partial_2^2 \phi + A v\cdot \grad \phi
     & \geq - \frac{f'(\Theta) \abs{\grad H}^2}{H} \rho - C {\norm{f'}_{L^\infty}}               \\
     & \geq \frac{\lambda_1\tilde C  \inf_{R_1}|f'(\Theta)|}{d_1} - C {\norm{f'}_{L^\infty}} \,.
  \end{align*}
  By~\eqref{AssumptionOnf}, we could then pick $d_1$ small, independent of $\epsilon$, to make the following hold
  \begin{equation*}
    - \partial^2_1 \phi - \epsilon \partial_2^2 \phi + A v\cdot \grad  \phi > 1
  \end{equation*}
  in $R_1^a$.

  On the other hand, in $R_1^b$, we have $\abs{\grad H(z_1,z_2)}^2 = z_1^2 + z_2^2$. Therefore, by Cauchy-Schwarz inequality,
  \begin{equation} \label{ine:cauchy}
    \abs[\Big]{ f'(\Theta) \frac{\abs{ \grad H}^2}{H} }
    =
    -  f'(\Theta) \frac{\abs{ \grad H}^2}{H} = -  f'(\Theta) \frac{z_1^2 + z_2^2}{z_1 z_2} \geq 2  \inf_{R_1}|f'|  \,.
  \end{equation}
  Also, note that in $R_1^b$ it holds that $\abs{\partial_i \Theta \partial_i H } = \paren{\partial_i H}^2$ for $i=1,2$.
  Thus, by~\eqref{AssumptionOnf}--\eqref{ine:cauchy} and $H\geq 1/\sqrt{A}$,
  we choose $f$ such that $\lambda_1 \inf_{R_1} \abs{f'} > 2$ and $\epsilon$ small enough to get
  \begin{align*}
     & - \partial^2_1 \phi - \epsilon \partial_2^2 \phi + A v\cdot \grad \phi \\
     & \geq
    -  \frac{f'(\Theta) \abs{\grad H}^2 }{ H} \rho
    -\frac{2}{A} \paren[\Big]{ \frac{f'(\Theta)}{H^2}\paren[\big]{ \partial_1 \Theta \partial_1 H + \epsilon \partial_2 \Theta\partial_2 H } }
    - \frac{C}{\sqrt{A}}                                                               \\
     & \geq
    -  \frac{f'(\Theta) \abs{\grad H}^2 }{ H} \rho
    -\frac{2}{A}\abs[\Big]{\frac{f'(\Theta)\abs{\grad H}^2}{H^2}} - \frac{C}{\sqrt{A}} \\
     & =
    \abs[\Big]{\frac{f'(\Theta)\abs{\grad H}^2}{H}}
    \paren[\Big]{ \rho - \frac{2}{AH}  } - \frac{C}{\sqrt{A}}
    \\
     & \geq \lambda_1\inf_{R_1} \abs{f'}  - \frac{C}{\sqrt{A}} >1  \,.
  \end{align*}

  Thus, we have just shown that there exists a function $f$ that satisfies~\eqref{AssumptionOnf} so that in $R_1$,
  \begin{equation*}
    - \partial^2_1 \phi - \epsilon \partial_2^2 \phi + A v\cdot \grad \phi >1 \,.
  \end{equation*}

  \step
  In $R_2$, there exist constants $C_1, C_2$ so that
  \begin{equation*}
    0<C_2 \leq C_1\abs{\grad H}^2 \leq  (\partial_1 H)^2 \,.
  \end{equation*}

  We then look at
  \begin{align*}
     & - \partial^2_1 \phi - \epsilon \partial_2^2 \phi + A v\cdot \grad \phi                                                                  \\
     & \quad \geq \frac{S}{d_1} \partial^2_1 H (\ln H +1) + \frac{S}{d_1}\frac{(\partial_1 H)^2}{H} -  \frac{f'(\Theta) \abs{\grad H}^2 }{ H} \rho  - C \\
     & \quad  \geq \frac{S}{d_1} \partial^2_1 H (\ln H +1) + \frac{S}{d_1}\frac{C_1 \abs{\grad H}^2}{H}
    - \lambda_2 \norm{f'}_{L^\infty(R_2)} \frac{ \abs{\grad H}^2 }{ H}   -  C                                                                           \\
     & \quad \geq  \frac{C_2}{C_1d_1}\paren[\Big]{ \frac{S C_1}{d_1}  - \lambda_2 \norm{f'}_{L^\infty(R_2)}   } - C \,.
  \end{align*}
  Pick $d_1$ smaller if needed to get
  \begin{equation*}
    - \partial_2 \phi - \epsilon \partial_2^2 \phi + A v\cdot \grad \phi > 1
    \quad \text{ in } R_2
    \,.
  \end{equation*}

  \step
  Combining Steps 3 and 4, we have shown that there exists a function $f$ such that
  \begin{equation*}
    - \partial_2 \phi - \epsilon \partial_2^2 \phi + A v\cdot \grad \phi > 1  \quad \text{ in } \Lambda \,.
  \end{equation*}
  By construction, $\phi \geq U_\epsilon$ on $\set{ H = d_1 } \cup \set{H=\frac{ 1  }{\sqrt{A}}}$. The comparison principle then tells us that
  \begin{equation*}
    \phi \geq U_\epsilon \quad \text{ in } \Lambda
  \end{equation*}
  as desired.
\end{proof}

\section{Proof of Lemma~\ref{lem:innerEst}}
\label{a:innerEst}
In this section, we give the proof of Lemma~\ref{lem:innerEst}.
This fact has been obtained in more generality by PDE method by Ishii and Souganidis~\cite{IshiiSouganidis12}.
Our method proof, still PDE-based, is different than that in~\cite{IshiiSouganidis12}.
Although the argument is new for our particular situation, it is an adaptation of the method in~\cite{Kumagai18}, where the author studies the Freidlin problem for first order Hamilton-Jacobi equations.

It is convenient to work in the so-called curvilinear coordinates $(h, \theta)$, in one cell.
Let $\mathcal Q_0^* = (0, 1)^2 - \Gamma_0$, where $\Gamma_0$ is the closure of one trajectory of the gradient flow of $H$ starting on the boundary of the unit square.
On $\mathcal Q_0^*$ we define the curvilinear coordinates by setting $h = H(x)$, $\theta = \Theta(x)$, where~$\Theta$ solves
\begin{equation*}
    \grad \Theta \cdot \grad H = 0 \,,
\end{equation*}
in $\mathcal Q_0^*$, normalized so that the range of $\Theta$ is $(0, 2\pi)$.
In this coordinate system, $h(x)$ determines the level set of the Hamiltonian to which  $x$ belongs and $\theta$ describes the position of $x$ on this level set.
Since $\grad \Theta$ and $\grad^\perp H$ are parallel, there must exist a non-zero function $\rho$ such that 
\begin{equation*}
  \grad \Theta = \rho \grad^\perp H\,.
\end{equation*}
By reversing the orientation of $\Theta$ if needed, we may assume, without loss of generality, that $\rho > 0$.
Let $J= \partial_1 H \partial_2 \Theta - \partial_2 H \partial_1 \Theta$ be the Jacobian of the coordinate transformation, and note
\begin{equation*}
  J = \rho \abs{\grad H}^2 \,, \quad  \abs{\grad \Theta} = \rho \abs{\grad H} \,.
\end{equation*}

Let $\gamma$ be the solution to~\eqref{e:gamma} with $\gamma_0 = x$, and $T$ be the time period of~$\gamma$.
Note $T$ only depends on $h = H(x)$, and is given by
\begin{equation} \label{eq:Gamma}
  T(h) \defeq \inf\{ t>0: \gamma(t,x)=x\}
    = \oint_{\set{H = h}} \frac{1}{\abs{\grad H}} \, \abs{d\ell} \,,
\end{equation}
where $\abs{d\ell}$ denotes the arc-length integral along the curve $\set{H = h}$.

Let $S(x)  \defeq  \inf\{t \st \gamma(t,x) \in \Gamma_0\}$ be the amount of time $\gamma$ takes to to reach $\Gamma_0$ starting from $x$.
This time is    not a continuous function of $x$.
Therefore, in order to make it continuous, we    modify it to the following continuous function
\begin{equation} \label{eq:gammatilde}
  \tilde{S}(x):= \begin{dcases}
    S(x)              & \text{ if } S(x) > \Gamma(H(x))/2, \\
    -S(x) + \Gamma(H(x)) & \text{ if } S(x) < \Gamma(H(x))/2.
  \end{dcases}
\end{equation}

As we have restricted our attention to one cell, we can assume $H \in [0,1]$. Define the coefficients $D_1$ and $D_2$ on $[0,1]$ 
as follows
\begin{subequations}
  \begin{align}
    \label{e:ADef}
    D_1(h)
     & = \frac{1}{T(h)}
    \oint_{\set{H = h}} \frac{ \abs{\partial_1 H}^2}{\abs{\grad H}}\, \abs{d\ell} \,,
    \\
    \label{e:BDef}
    D_2(h)
     & = \frac{1}{T(h)}
    \oint_{\set{H = h}} \frac{\partial_1^2 H}{\abs{\grad H}} \, \abs{d\ell}\,.
  \end{align}
\end{subequations}
Note that by Gauss--Green theorem, we have
\begin{equation*}
  T(h) D_1(h)  = - \int_{\set{ H\geq h}} \partial_1^2 H(x) \, dx 
  = \int_1^h \oint_{\set{H=h}} \frac{ \partial_1^2 H }{\abs{\grad H}} \, \abs{d\ell}\, dh \,.
\end{equation*}
Therefore,
\begin{equation} \label{e:D1D2relation}
  \frac{d}{dh} (T(h) D_1(h)) = T(h) D_2(h)\,.
\end{equation}

We are now ready to show the proof of Lemma~\ref{lem:innerEst}.

\begin{proof}[Proof of Lemma~\ref{lem:innerEst}]
  As before, we restrict our attention to a particular cell $(0,1)^2$ as the estimate is the same for other ones.
  \restartsteps
  \step
  Let $U_\epsilon(x) \defeq \E^x \tau_0^\epsilon$ and $\Omega_\epsilon \defeq (0, 1)^2 - \cH_{\alpha}$.
  Then, $U_\epsilon$ is the solution to the equation
  \begin{equation*}
    -\frac{1}{2}\partial_1^2 U_\epsilon - \frac{\epsilon^2}{2} \partial_2^2 U_\epsilon + A
    v\cdot \grad U_\epsilon= 1 \quad \text{ on } \Omega_\epsilon \,, \\
  \end{equation*}
  with boundary condition
  \begin{equation*}
    U_\epsilon = 0  \quad  \text{ on } \partial \Omega_\epsilon \,.
  \end{equation*}
  Lemma~\ref{lem:innerEst} will follow immediately from the uniform bound
  \begin{equation*}
    \sup_\epsilon \norm{ U_\epsilon}_{L^\infty(\Omega_1^\epsilon)} \leq C.
  \end{equation*}
  To see why this bound is true, let us consider the solution $\bar U$ to the ODE
  \begin{equation*}
    \begin{dcases}
      -D_1(h) \partial_h^2 \bar U - D_2(h) \partial_h \bar U & = 4 \,, \\
      \hspace{3cm} \bar U(0)                                      & = 4 \,.
    \end{dcases}
  \end{equation*}
  Note that $\bar U$ is bounded. To see this, we use~\eqref{e:D1D2relation} to rewrite the equation 
  \begin{equation*}
    - \frac{1}{T(h)} \partial_h \paren[\Big]{ T(h) D_1 (h) \partial_h \bar U  } = 4 \,.
  \end{equation*}
   Observe that $T(h) D_1(h) \approx O(1-h)$ and $T(h)\to T_0 >0$ as $h\to 1$;
   $T(h) \approx O(\abs{\ln h})$ and $D_1(h) \approx O(1/\abs{\ln h})$ as $h\to 0$ (see Chapter 8.2 in \cite{FreidlinWentzell12}).
   Using these asymptotics, we deduce
  \begin{equation*}
    \partial_h \bar U(h) = \frac{4}{T(h)D_1(h)} \int_h^1 T(s) \,  ds \,,
    \qquad
    \bar U(h) = \int_0^h \frac{4}{T(s) D_1(s)} \int_s^1 T(r) \, dr  ds \,,
  \end{equation*}
  and
  \begin{equation*}
    \norm{\bar U}_{W^{1,\infty}} \leq C \,.
  \end{equation*}

  \step
  Note that $\bar U \circ H$ is a function on $\Omega$.
  Let 
  \begin{equation*}
    g = \partial_1^2 (\bar U \circ H) \,,
  \end{equation*}
  and we see that
  \begin{equation*}
    \bar g(x) \defeq \frac{1}{T(H(x))}
    \int_0^{T(H(x))} g(\gamma(t,x)) \, dt = -4 \,,
  \end{equation*}
  where $T$ is defined in \eqref{eq:Gamma}.
  Define
  \begin{equation*}
    \varphi(x) = \int_0^{\tilde S(x)} ( \bar g(x)- g(\gamma_t(x))) \,
    dt \,,
  \end{equation*}
  where $\tilde S$ is defined in \eqref{eq:gammatilde}.
  Note that %
  \begin{equation} \label{eq:vDotGradVarphi}
    v(x)\cdot \grad \varphi (x) = g(x) - \bar g(x) = g(x) + 4 \,.
  \end{equation}
  To see this, consider
  \begin{align*}
    \varphi(\gamma(s,x)) & = -\int_0^{\tilde S(\gamma(s,x))}
    \paren[\Big]{g(\gamma(t,\gamma(s,x))) - \bar g(\gamma(s,x))} \, dt                               \\
                         & = -\int_{s}^{\tilde S(x)} \paren[\Big]{ g(\gamma(t,x)) - \bar g(x)} \, dt
    \,.
  \end{align*}
  Differentiate in $s$ and evaluate at $s=0$, we get \eqref{eq:vDotGradVarphi}.

  \step
  Let
  \begin{equation*}
    G_\epsilon \defeq \bar U\circ H + \frac{1}{A}\varphi \,,
    \qquad
    L_\epsilon = -\frac{1}{2} \partial_1^2 - \frac{\epsilon^2}{2} \partial_2^2 + A v \cdot \grad\,,
  \end{equation*}
  and note
  \begin{equation*}
    \begin{aligned}
      L_\epsilon G_\epsilon
      &  =  - \frac{1}{2}\partial_1^2( \bar U \circ H) -\frac{1}{2 A} \partial_1^2 \varphi -
      \frac{\epsilon^2}{2} \partial_2^2 (\bar U \circ H)  
       - \frac{\epsilon^2}{2 A}
      \partial_2^2 \varphi + g(x)  + 4
      \\
      & =  -\frac{1}{2 A} \partial_1^2 \varphi
      - \frac{\epsilon^2}{2} \partial_2^2 (\bar U \circ H)
      - \frac{\epsilon^2}{2 A} \partial_2^2 \varphi + 4
      = e_\epsilon + 4\,,
    \end{aligned}
  \end{equation*}
  where $e_\epsilon \defeq  -\frac{1}{2 A} \partial_1^2 \varphi
    - \frac{\epsilon^2}{2} \partial_2^2 (\bar U \circ H)
    - \frac{\epsilon^2}{2 A} \partial_2^2 \varphi$.
  Since $U$ is smooth
  and $e_\epsilon$ converge uniformly to $0$ as $\epsilon \to 0$,
  there exists an $\epsilon_0$ such that for all $\epsilon \leq
    \epsilon_0$,
  $L_\epsilon G_\epsilon \geq 1$ and
  $G_\epsilon \geq U_\epsilon$ on $\partial \Omega_\epsilon$.
  By the maximum principle,
  $G_\epsilon \geq U_\epsilon$ on $\Omega_\epsilon$.
  Finally, observe that $\sup_\epsilon \norm{G_\epsilon}_{L^\infty} <\infty$, which implies what we want.
\end{proof}

\section{Upper bound for energy constrained flows (Theorem~\ref{t:energy})}\label{sec:energy}

In this section our aim is to prove the upper bound in Theorem~\ref{t:energy}.
As in the proof of Proposition~\ref{p:enstrophy}, we will consider the doubled strip $S_2 = \R \times(-1, 1)$ with Dirichlet boundary conditions, and only use velocity fields~$v$ satisfying~\eqref{e:vSymm}.
Our aim is to find $v \in \mathcal V^{0, p}_\U$ satisfying~\eqref{e:vSymm} such that
\begin{equation*}
  \norm{T^v}_{L^\infty} \leq \frac{C\ln \U}{\U}\,,
\end{equation*}
for all sufficiently large~$\U$.
The flow we use is an analog of the one used by Marcotte et\ al.~\cite{MarcotteDoeringEA18} adapted to the periodic strip, and is shown in Figure~\ref{f:skew}.
It consists of $1/\epsilon$ convection rolls of width~$\epsilon$, height~$1$ skewed so that the center of the roll is only $\delta$ away from the top boundary.
Here $\epsilon, \delta > 0$ are small numbers that will shortly be chosen in terms of the P\'eclet number~$\U$.

Let $\nu \in (0,1)$, $\delta = \epsilon^{2 + \nu}$ and
$H \colon \R^2 \to \R$ be defined by
\begin{equation*}
  H(x_1, x_2)  \defeq H_1(x_1) H_{2}(x_2)\,,
\end{equation*}
where $H_1 \colon \R \to \R$, $H_2 = H_{2, \epsilon} \colon [0,1]\to \R$ are Lipschitz functions such that
\begin{equation*}
  H_1(x_1 + 2) = H_1(x_1) \,,
  \qquad
  H_2( - x_2) = H_2(x_2)\,,
\end{equation*}
\begin{equation*}
  H_1(x_1) = \begin{dcases}
    x_1 & x_1 \in \Bigl[0, \frac{1}{2} \Bigr) \,, \\
    1 - x_1 & x_1 \in \Bigl[\frac{1}{2}, \frac{3}{2} \Bigr) \,, \\
    -2 + x_1 & x_1 \in \Bigl[\frac{3}{2}, 2 \Bigr) \,.
  \end{dcases}
\end{equation*}
and
\begin{equation*}
 H_{2}(x_2) = 
 \begin{dcases}
    x_2 & x_2  \in  [-1+2\delta, 1- 2\delta] \,, \\
    0 & x_2 = \pm 1 \,.
 \end{dcases}
\end{equation*}
Moreover, we assume $H_1$, $H_2$ are such that $H$ has only one non-degenerate critical point in the square~$(0, 2) \times (0, 1)$.
Stream lines of such a Hamiltonian are shown in Figure~\ref{f:skew}.

Given~$\epsilon > 0$, define the rescaled Hamiltonian~$H^\epsilon$ by
\begin{equation*}
   H^{\epsilon} (x_1,x_2) \defeq H\paren[\Big]{ \frac{x_1}{\epsilon},  x_2 } \,,
  \quad \text{ and set} \quad
  v^{\epsilon} \defeq \frac{A_\epsilon}{\epsilon} \grad^\perp H^\epsilon = \frac{A_\epsilon }{\epsilon}
  \begin{pmatrix}
  \partial_2 H^\epsilon \\ - \partial_1 H^\epsilon
      \end{pmatrix}\,.
\end{equation*}
Let $T_{\epsilon} = T^{v^{\epsilon}}$ be the solution to~\eqref{e:T}--\eqref{e:TBC} with drift~$v^\epsilon$.

By uniqueness of solutions we see that $T_{\epsilon}$ satisfies $T_{\epsilon}( x_1 + 2 \epsilon, x_2) = T_{\epsilon}(x_1, x_2)$.
Thus, we change variables and define
\begin{equation*}
  y_1 = \frac{x_1}{\epsilon}\,,
  \quad
  y_2 = x_2\,,
  \qquad\text{and}\qquad
  v = \grad^\perp_y H\,.
\end{equation*}
In these coordinates we see that $T_{\epsilon}$ satisfies
\begin{subequations}
\begin{equation}\label{e:Tepsimproved}
    A_\epsilon v \cdot \grad_y T_{\epsilon}
  - \frac{1}{2}\partial_{y_1}^2 T_{\epsilon}
  -  \frac{1}{2}\epsilon^2\partial_{y_2}^2 T_{\epsilon} = \epsilon^2 \,,
\end{equation}
with boundary conditions
\begin{equation}\label{e:TepsBC}
  T_\epsilon(y_1 + 2, y_2) = T_\epsilon (y_1, y_2)\,,
  \qquad\text{and}\qquad
  T_\epsilon(y_1, 1) = T_\epsilon (y_1, -1) = 0\,.
\end{equation}
\end{subequations}

To estimate the size of~$T_\epsilon$, consider the associated diffusion let $Z^\epsilon = (Z^\epsilon_1, Z^\epsilon_2)$ which solves the SDE~\eqref{eq:degenerateProcess}.
And let~$\tau^\epsilon$ (defined in~\eqref{e:tauDef}) be the exit time of $Z$ from the doubled strip $S_2$.
By the Dynkin formula, we know $T_\epsilon = \epsilon^2 \E \tau^\epsilon$, and so estimating~$\E \tau^\epsilon$ will give us a bound on $T_\epsilon$.
This is our next proposition.

\begin{proposition}\label{p:tauBdupdate}
  Given a Hamiltonian~$H$ in the above form, choose $A_\epsilon = 1/\epsilon^\nu$, $v = \grad^\perp_y H$.
  There exists a constant $C = C(\nu)$ such that
  \begin{equation}
    \label{e:tauBdupdate}
    \sup_{z \in \Omega'} \E^z \tau^\epsilon \leq   \frac{C \abs{\ln \epsilon}}{A_\epsilon} \,.
  \end{equation}
  for all sufficiently small~$\epsilon$.
\end{proposition}

The reason the bound~\eqref{e:tauBdupdate} is as follows.
In time $O(\abs{\ln \epsilon}/A_\epsilon)$, deterministic trajectories of the flow~$v$ will move most interior points to $O(\delta)$ away from the $\partial_D S_2$.
In this region, 
the drift has speed $O(A_\epsilon/\delta)$ so particles in this region have $O(\delta/A_\epsilon)$ time to diffuse vertically before getting carried away from the boundary $\partial_D S_2$.
Within this time,
particles can diffuse a vertical distance of $O(\epsilon  \sqrt{\delta/A_\epsilon})$.
By choice of~$\delta = \epsilon^2 / A_\epsilon$, and so $\epsilon \sqrt{\delta / A_\epsilon} = \delta$, and hence particles a distance $O(\delta)$ away from $\partial_D S_2$ exit $S_2$ with non-zero probability, before being carried away from~$\partial_D S_2$ by the flow.
Now using the strong Markov property we can estimate~$\E \tau^\epsilon$ by the expected time to success of repeated Bernoulli trials, leading to~\eqref{e:tauBdupdate}.
Before carrying out these details, we first show how it can be used to finish the proof of Theorem~\ref{t:energy}.
\begin{proof}[Proof of the upper bound in Theorem~\ref{t:energy}]
  Clearly it is enough to prove~\eqref{e:E1upperBdLinf} for $q = \infty$.
  Let $v$ be the flow from the Hamiltonian in Proposition~\ref{p:tauBdupdate} and $A_\epsilon = \epsilon^{-\nu}$.
  We note that
  \begin{equation*}
    \U = \norm{v^\epsilon}_{L^p}
    = O\paren[\Big]{\frac{A_\epsilon}{\epsilon} \paren[\Big]{ \frac{1}{\epsilon^p} + \frac{1}{\delta^{p-1}} }^{1/p} }
    = O\paren[\Big]{\frac{A_\epsilon}{\epsilon} \paren[\Big]{ \frac{1}{\epsilon^p} + \frac{1}{\epsilon^{(2+\nu)(p-1)}} }^{1/p} }
    = O(\epsilon^{-q})\,,
  \end{equation*}
  where
  \begin{equation*}
    p' = \begin{dcases}
      2 + \nu & 1 \leq p \leq \frac{2 + \nu}{1 + \nu} \,,
      \\
      1 + \nu + \frac{(2 + \nu)(p -1)}{p}
	& p \geq \frac{2 + \nu}{1 + \nu}\,.
    \end{dcases}
  \end{equation*}

  Let~$T_\epsilon$ be the solution to~\eqref{e:Tepsimproved}--\eqref{e:TepsBC}, and note that by Dynkin's formula, $T_\epsilon = \epsilon^2 \E \tau^\epsilon$.
  Thus, by Proposition~\ref{p:tauBdupdate}
  \begin{equation*}
    \norm{T_\epsilon}_{L^\infty} 
    \leq  \frac{C \epsilon^2 {\abs{\ln \epsilon}}}{A_\epsilon}
    \leq  \frac{C \ln \U}{\U^{(2 + \nu)/ p'}}\,.
  \end{equation*}
  If $p < 2$, then by choosing $\nu > 0$ small enough we can ensure $p \leq (2 + \nu) / (1 + \nu)$.
  In this case $2 + \nu = p'$ and hence
  \begin{equation*}
    \norm{T_\epsilon}_{L^\infty} \leq \frac{C \ln \U}{\U}\,.
  \end{equation*}
  On the other hand, if $p \geq 2$, then for any $\mu > 0$ we can choose $\nu > 0$ small enough to ensure
  \begin{equation*}
    \norm{T_\epsilon}_{L^\infty} \leq \frac{C_\mu \ln \U}{\U^{\frac{2 p}{3p - 2} - \mu}}\,,
  \end{equation*}
  finishing the proof.
\end{proof}

It remains to prove Proposition~\ref{p:tauBdupdate}.
The key step is to show that starting from any point in~$S_2$, the probability $Z^\epsilon$ hits the boundary $\partial_D S_2$ in time $O(\abs{\ln \epsilon} / A_\epsilon)$ is bounded away from~$0$.
This is our next lemma.
\begin{lemma}\label{l:optimallowerexit}
 Let $A_\epsilon = \epsilon^{-\nu}$. There exists constants $ p_0 = p_0(\nu) \in (0,1)$ and $K = K(\nu) \in \N$, independent of $\epsilon$, such that
 \begin{equation}
   \label{e:optimalexitprobability}
  \inf_{z\in \Omega'}\P^z \paren[\Big]{ \tau^\epsilon \leq \frac{K\abs{ \ln \epsilon}}{A_\epsilon}  } 
  \geq p_0 \,,
 \end{equation} 
  for all sufficiently small~$\epsilon > 0$.
\end{lemma}

Using Lemma~\ref{l:optimallowerexit} one can prove Proposition~\ref{p:tauBdupdate} by treating the exit from the strip as repeated Bernoulli trials.
\begin{proof}[Proof of Proposition~\ref{p:tauBdupdate}]
    Letting $t_i = i K \abs{\ln\epsilon}/A_\epsilon$, we note
    \begin{multline*}
      \sup_{z \in \Omega'} \P^z( \tau^\epsilon \geq t_i )
	= \sup_{z \in \Omega'} \E^z( \E^z(\one_{\tau^\epsilon \geq t_{i-1} } \one_{\tau^\epsilon \geq t_i} \given \mathcal F_{t_{i-1}} ) )
      \\
	= \sup_{z \in \Omega'} \E^z( \one_{\tau^\epsilon \geq t_{i-1} } \P^{Z_{t_{i-1}}} (\tau^\epsilon \geq (t_i - t_{i-1}) )
	\leq (1 - p_0) \sup_{z \in \Omega'} \P^z ( \tau^\epsilon \geq t_{i-1} )\,.
    \end{multline*}
    and hence
    \begin{equation*}
      \sup_{z \in \Omega'} \P^z( \tau^\epsilon \geq t_i )
	\leq (1 - p_0)^i\,.
    \end{equation*}
    Consequently,
  \begin{align*}
    \E^z \tau^\epsilon 
      &= 
      \int_0^\infty \P^z\paren{\tau^\epsilon \geq t} \, dt  
      \leq \sum_{i=0}^\infty (t_{i+1} - t_i) \P^z(\tau^\epsilon \geq t_i)
    \\
      &\leq
      \frac{K \abs{\ln \epsilon}}{A_\epsilon}
      \sum_{i=0}^\infty (1 - p_0)^i
      = \frac{K  \, \abs{\ln\epsilon}}{p_0 A_\epsilon} \,,
  \end{align*}
  for every $z \in \Omega'$.
  This yields~\eqref{e:tauBdupdate} as desired.
\end{proof}

It remains to prove Lemma~\ref{l:optimallowerexit}, and this constitutes the bulk of this section.
We will subsequently assume $A_\epsilon = \epsilon^{-\nu}$, and for notational convenience simply write $A$ instead of $A_\epsilon$.

Let $\kappa_1$, defined by
\begin{equation}
  \kappa_1 \defeq \inf \set[\big]{ t\geq 0 \st Z^\epsilon_t \in  (0,2)\times (1-2  \delta, 1) } \,,
\end{equation}
be the first time $Z^\epsilon_t$ hits the set $(0,2)\times (1-2\delta, 1)$.
\begin{lemma}
  \label{l:optimalinteriorestbulk}
  Let $0< h_0 \ll c_0$ be a small constant independent of $\epsilon$, and define
  \begin{equation*}
    R_{h_0} = \Omega \cap \paren[\big]{ \cH_{h_0}^c \cup (1-c_0, 1+c_0)\times (c_0, 1-c_0) }\,.
  \end{equation*}
  Suppose~$h_0$ is small enough so that $\mathcal B_{h_0}^c \cap (1 - c_0, 1 + c_0) \times (c_0, 1 - c_0)$ is nonempty.
  There exists constants $C_0>0$ and $p_1 \in (0,1)$ such that 
  \begin{equation}
  \label{ine:hittingtopprobabilitybulk}
    \inf_{z_0 \in R_{h_0}}\P^{z_0} \paren[\Big]{ \kappa_1 \leq \frac{C_0}{A}   } \geq p_1 \,.
  \end{equation}
\end{lemma}
The proof of Lemma~\ref{l:optimalinteriorestbulk} is
based on a standard tube lemma argument and is
presented in Appendix~\ref{s:tubelemmas}.

\begin{lemma}\label{lem:lambda2estop}
  Let $h_0$ be as in Lemma~\ref{l:optimalinteriorestbulk}, 
  $T_0 = \inf\set{t>0: \gamma_{2,t} \in \set{2c_0, 1-2c_0}   }$, 
  and $T_1 = \min\set{T_0, \abs{\ln A}/A}$.
  Then
  \begin{equation} \label{e:lambda2estop}
    \inf_{\cH_{h_0} \cap (0,2)\times (0,c_0)}\P^{z_0}\paren[\Big]{
      Z_{T_1} \in (1-2c_0,1+ 2c_0)\times \paren{c_0, 1-c_0} }
    \geq 1  - \frac{C \ln A}{A^{1/2}}\,,
  \end{equation}
  and
  \begin{equation} \label{e:lambda2estop2}
    \inf_{\cH_{h_0} \cap (0,2)\times(1-c_0,1-2\delta)}\P^{z_0}\paren[\Big]{
      Z_{T_1} \in \paren[\big]{(0, 2c_0)\cup (2-2 c_0, 2)} \times \paren{ c_0, 1-c_0} }
    \geq 1 - \frac{C \ln A}{A^{1/2}}\,.
  \end{equation}
\end{lemma}
\begin{proof}
  We only show the proof for~\eqref{e:lambda2estop} as~\eqref{e:lambda2estop2} holds also by symmetry.
  Let $q \geq 2$ be some large number to be chosen later, and let $\tilde z_0$  be the point in the set $\set{ H \in (A^{-1/q}, h_0)  }$ which is closest to $z_0$.
  Let $\tilde d = A \abs{z_0 - \tilde z_0}$
  and $\gamma_t$ be the solution to~\eqref{e:gamma}, with~$\gamma_0 = \tilde z_0$.
  Note that, if $z_0$ is already in $\set{ H \in (A^{-1/q}, h_0)  }$, then $\tilde d=0$.
  Also, by Assumption~\ref{A1},
  \begin{equation} \label{ine:projectionDistanceop}
      \frac{\tilde d}{A} \leq  \frac{C}{A^{1/(2q)}} \,.
  \end{equation}
  By It\^{o} formula, we have
  \begin{equation*}
      \E^{z_0}\abs{Z_t - \gamma_t}^2
  \leq \frac{\tilde d^2}{A^2}
    + 2 A \norm{v}_{C^1}
      \int_0^t \E^{z_0}\abs{Z_s - \gamma_s}^2 \, ds
      +(1 + \epsilon^2)t \,.
  \end{equation*}
  By Gronwall's inequality, it follows that
  \begin{equation*}
    \begin{aligned}
      \E^{z_0}\abs{Z_t - \gamma_t}^2
      \leq \paren[\Big]{\frac{\tilde d^2}{A^2} + (1+\epsilon^2)t}e^{2 \norm{v}_{C^1} At}\,.
    \end{aligned}
  \end{equation*}
  Now, let $T = \inf\set{ t >0 \st \gamma_{2,t} \in \paren{2c_0, 1-2c_0}}$, and note that $T \leq D \ln A / (Aq)$ for some constant $D > 0$.
  By~\eqref{ine:projectionDistanceop}, we have
  \begin{equation*}
    \begin{aligned}
       \P^{z_0}\paren[\Big]{ \abs{Z_T - \gamma_T} \geq \frac{c_0}{10}  }
      &\leq
      \frac{100}{c_0^2}\paren[\Big]{\frac{C}{A^{2q}} + (1+\epsilon^2)\frac{D\ln A}{Aq}}e^{2 \norm{v}_{C^1} D\ln A/q} \\
       & \leq
      C A^{2D \norm{v}_{C^1} /q-1}\ln A
	\leq \frac{C \ln A}{A^{1/2}}\,,
    \end{aligned}
  \end{equation*}
  provided $q$ is chosen so that $ 2\norm{v}_{C^1} D/q -1 < -1/2$.
  we have
  \begin{equation} \label{ine:lowerProjectionop}
    \P^{z_0}\paren[\Big]{ \abs{Z_T - \gamma_T} < \frac{c_0}{10}  }
    \geq
    1 - \frac{C \ln A}{A^{1/2}} 
    \,.
  \end{equation}
  Since the trajectories of~$Z$ are continuous,
  \begin{equation*}
    \set{Z_{T_1} \in (1-2c_0,1+ 2c_0)\times \paren{c_0, 1-c_0}}
    \supseteq \set[\Big]{\abs{Z_T - \gamma_T} < \frac{c_0}{10} }\,,
  \end{equation*}
  from which~\eqref{e:lambda2estop} follows.
\end{proof}

\begin{lemma}
  \label{lem:boundaryestimateoptimal}
  There exists constants $D >0$, $p_2 \in (0, 1)$, independent of $\epsilon$ so that 
  \begin{equation}
    \label{ine:boundaryestimateoptimal}
    \inf_{z_0 \in \cH_{h_0}}\P^{z_0}\paren[\Big]{ \kappa_1 \leq \frac{D\abs{\ln A}}{A}  } \geq p_2 \,.
  \end{equation}
\end{lemma}
\begin{proof}
  Denote
  \begin{gather*}
    \Box_1 \defeq (1-2c_0,1+2c_0)\times (c_0, 1-c_0)  \,, \\
    \Box_2 \defeq \cH_{h_0} \cap \set{x_2 \in (0, c_0)} \,, \\
    \Box_3 \defeq \cH_{h_0} \cap \paren[\big]{(0,2c_0) \cup (2-2c_0)}\times (c_0, 1-c_0) \,, \\
    \Box_4 \defeq \cH_{h_0} \cap \set{ x_2 \in (1-c_0, 1)}  \,.
  \end{gather*}
  First, if $z_0 \in \cH_{h_0} \cap \Box_1$, we are done, by Lemma~\ref{l:optimalinteriorestbulk}.

  Suppose now that $z_0\in \Box_2$.
  Let $T_1$ be as in Lemma~\ref{lem:lambda2estop}.
  By Lemmas~\ref{l:optimalinteriorestbulk}, \ref{lem:lambda2estop} and the strong Markov property we note
  \begin{align}
    \nonumber
    \P^{z_0} \paren[\Big]{ \kappa_1 \leq \frac{D}{A} + T_1  }
      &\geq \P^{z_0}\paren[\big]{Z_{T_1} \in \Box_1}
      \inf_{z_1 \in \Box_1} \P^{z_1}\paren[\Big]{ \kappa_1 \leq \frac{D}{A}  }
    \\
    \label{ine:box2est}
    &\geq
      \paren[\Big]{1 - \frac{C \ln A}{A} }p_1\,.
  \end{align}

  Suppose now that $z_0 \in \Box_3$. 
  Denote $\kappa_2 \defeq \inf\set{t>0 \st Z_{1,t} \in \set{ 2c_0, 2-2c_0}}$.
  By a similar argument as in Lemma~\ref{lem:lambda2estop}, there exists $p \in (0,1)$ such that 
  \begin{equation*}
    \inf_{z_0 \in \Box_3} \P^{z_0} \paren[\Big]{ \kappa_2 \leq \frac{\abs{\ln A}}{A}  } \geq p \,.
  \end{equation*}
  There are two possibilities:
  \begin{enumerate}
    \item There exists a $p_2'$, independent of $\epsilon$ such that 
      \begin{equation*}
	\P^{z_0}\paren[\Big]{Z_{\kappa_2} \in \Box_2 \,; \kappa_2 \leq \frac{\abs{\ln A}}{A} } \geq p_2'\,.
      \end{equation*}
      In this case, we can apply the same argument as in~\eqref{ine:box2est} to arrive at the desired result.

    \item
      Otherwise, there exists a constant $p_2'$, independent of $\epsilon$ such that 
      \begin{equation*}
        \P^{z_0}\paren[\Big]{H(Z_{\kappa_2}) \geq h_1 \,;
        \kappa_2 \leq \frac{\abs{\ln A}}{A}  } \geq p_2' \,,
      \end{equation*}
      for some $h_1$ independent of $\epsilon$.
      We can then apply Lemma~\ref{l:optimalinteriorestbulk} to get the desired result.
  \end{enumerate}

  The same argument works when $z_0 \in \Box_4$,  and this completes the proof of~\eqref{ine:boundaryestimateoptimal}.
\end{proof}

\begin{lemma}
 \label{lem:exittopprobability} 
 There exists a constant $p_3 \in (0,1)$ such that 
 \begin{equation}
 \label{ine:exittopprobability} 
   \inf_{z_0 \in \set{z\st z_2 \geq 1- 2\delta}}\P^{z_0} \paren[\Big]{ \tau^\epsilon \leq \frac{\epsilon}{A}  } \geq p_3 \,.
 \end{equation}
\end{lemma}
\begin{proof}
  Denote $T_3(z) = \inf\set{ t>0 \st \gamma_{2,t} \leq 1-4\delta \,, \gamma_0 = z   }$,
  and let
  \begin{equation*}
  T_4 \defeq    \inf_{\set{z \st z_2 \geq 1- 2 \delta}} T_3(z)\,.
  \end{equation*}
  By definition of~$H$ we see that $T_4 \geq C \delta / A$ for some constant~$C$.
  In time $C \delta / A$ the process~$Z$ diffuses a distance of $O(\epsilon \sqrt{\delta/A}) = O(\delta)$ vertically, and hence should hit the top boundary with a probability that is bounded away from~$0$.
  That is, we should have
  \begin{equation}\label{e:exittopprob1}
    \P^{z_0}\paren[\big]{ \tau^\epsilon \leq T_4   } \geq p_3 \,,
  \end{equation}
  which immediately implies~\eqref{ine:exittopprobability}.
  The inequality~\eqref{e:exittopprob1} can proved using a tube lemma (Lemma~\ref{l:smallDev}) and is the same as the proof of Lemma~\ref{lem:topest}.
\end{proof}

\begin{proof}[Proof of Lemma~\ref{l:optimallowerexit}]
  Given Lemmas~\ref{l:optimalinteriorestbulk}, \ref{lem:boundaryestimateoptimal}, \ref{lem:exittopprobability}, the proof of~\eqref{e:optimalexitprobability} is identical to that of Lemma~\ref{lem:lowerBoundExit}.
\end{proof}

\appendix
\section{Tube Lemmas}\label{s:tubelemmas}
In this appendix, we prove several ``tube lemmas'' and estimate the probability a diffusion stays close to the underlying deterministic flow.
Many such estimates are standard and can be found in books (see for instance~\cites{FreidlinWentzell12}).
However, in our situation, we require estimates where the diffusion coefficient is degenerate in one direction and the amplitude of the drift is large.
While the proofs follow standard techniques, the estimates themselves aren't readily available in the literature, and we present them here.

Throughout this appendix we consider the SDE
\begin{equation}\label{eq:SDE}
  dZ_t = Av(Z_t) \, dt  + \sigma \, dB_t \,,
\end{equation}
where
\begin{equation} \label{cond:flow}
  \norm{v}_{L^\infty} \leq 1 \,, \quad  \norm{Dv}_{L^\infty} \leq 1   \,,
\end{equation}
\begin{equation} \label{cond:sigma}
  \sigma = (\sigma_{ij}) = \begin{pmatrix} 1 & 0\\ 0 & \epsilon \end{pmatrix}  \,.
\end{equation}
For notational convenience we will often denote the diagonal entries with just one subscript and write $\sigma_i$ for $\sigma_{ii}$ (i.e.\ $\sigma_1 = 1$ and $\sigma_2 = \epsilon$).

\begin{lemma} \label{l:STPC}
  Fix $\lambda, \beta > 0$, and define $T = T_{\beta, A}$ and $R = R_{A, \lambda}$ by
  \begin{equation}\label{e:TR}
    T \defeq \frac{\beta}{A}\,,
    \qquad
    R \defeq \paren[\Big]{ 1-\frac{\lambda}{\sqrt{A}},\, 1 + \frac{\lambda}{\sqrt{A}} }\times \paren{1-\epsilon, \, 1}\,.
  \end{equation}
  Let $z_0 \in R$, $u \in C^1(\R^2)$ and let $\tilde \gamma$ be the solution to the ODE
  \begin{equation*}
    \partial_t \tilde \gamma_t = A u(\tilde \gamma_t) \, dt \,,
    \qquad\text{with}\qquad
    \tilde \gamma_0 = z_0 \,,
  \end{equation*}
and $\tilde \Gamma = \set{ \tilde \gamma(t) \st t \in [0, T] }$ be the image of $\tilde \gamma$.
  Denote 
  \begin{equation*}
    L_T = \frac{A^2}{2}
      \int_0^T \sum_{i = 1,2}
      \paren[\Big]{
        \frac{\abs{u_i(\tilde \gamma(t)) - v_i(\tilde \gamma(t)) } }{\sigma_i} +
      \sum_{j = 1}^2  \frac{\sigma_j \norm{\partial_j v_i}_{L^\infty(R + \tilde \Gamma)}}{\sigma_i\sqrt{A}}
      }^2 \, dt \,.
  \end{equation*}
  Then for some $\alpha>0$  we have
  \begin{equation*}
    \P^{z_0} \paren[\Big]{
      \sup_{0 \leq t \leq T}
      \abs{\sigma^{-1} (Z_t - \tilde \gamma_t)}_{\infty}  \leq \frac{\lambda}{\sqrt{A}} }
    \\
    \geq
    \P \paren[\Big]{\sup_{t\leq T} \abs{B_t}_\infty \leq \frac{\lambda}{\sqrt{A}} }
    \exp\paren[\Big]{ 
      -\alpha\sqrt{L_T} - \frac{1}{2} L_T
    }
  \end{equation*}
  for all sufficiently large~$A$.
  Here the notation $\abs{z}_\infty$ denotes $\max_i \abs{z_i}$.
\end{lemma}
\begin{remark}
  A similar upper bound also holds, but is not needed for purposes of this paper.
\end{remark}
\begin{proof}
  Define the process $\tilde Z$ by
  \begin{equation*}
    d\tilde Z_t = A u(\tilde  \gamma_t) \, dt + \sigma \, dB_t\,,
    \qquad\text{with}\qquad \tilde Z_0 = z_0\,.
  \end{equation*}
  Define
  \begin{gather}
    h(t) \defeq A( u(\tilde \gamma_t) - v(\tilde Z_t))\,,  \nonumber
    \\
    \hat h(t) \defeq \sigma^{-1} h(t)\,, \nonumber
    \\
    M_t  \defeq \exp \paren[\Big]{ -\int_0^t \hat h(s) \, dB_s -  \frac{1}{2} \int_0^t \hat h(s)^2 \, ds   } \label{eq:martingale}
  \end{gather}
  \begin{equation*}
  \end{equation*}
  and a measure $\hat \P$ so that
  \begin{equation*}
    d\hat \P = M_T \, d\P \,.
  \end{equation*}
  By the Girsanov theorem (see, for example, Theorem 8.6.6 in \cite{Oksendal03}), the process
  \begin{equation*}
    \hat B_t \defeq \int_0^t \hat h(s) \, ds +   B_t
  \end{equation*}
  is a Brownian motion with respect to the measure $\hat \P$ up to time~$T$.
  Since
  \begin{equation*}
    d \tilde Z = A v(\tilde Z) \, dt + \sigma \, d\hat B_t\,,
  \end{equation*}
  by weak uniqueness we have
  \begin{equation*}
    \E^{z_0} f(Z_t) = \hat \E^{z_0} f(\tilde Z_t)
    = \hat \E^{z_0} f(\tilde \gamma_t + \sigma B_t)
    = \E^{z_0} \paren[\big]{ f(\tilde \gamma_t + \sigma B_t) M_t } \,,
  \end{equation*}
  for any test function~$f$.
  Thus
  \begin{align*}
    \P^{z_0} \paren[\Big]{
      \sup_{t \leq T}
      \abs{\sigma^{-1} (Z_t - \tilde \gamma_t)}_\infty
      \leq \frac{\lambda}{\sqrt{A}}
    }
     & = \E^{z_0} \paren[\Big]{ \one_K M_T  } \,.
  \end{align*}
  where
  \begin{equation*}
    K \defeq \set[\Big]{
      \sup_{t \leq T}\,
      \abs{B_t}_\infty \leq \frac{\lambda}{\sqrt{A}}
    }\,.
  \end{equation*}

  Now let~$\alpha = (2 / \P^{z_0}(K))^{1/2}$, and $\hat K$ be the event
  \begin{equation*}
    \hat K \defeq \set[\Big]{
      \paren[\Big]{ \int_0^T \hat h(t) \, dB_t }^2
      < \alpha^2 \int_0^T \hat h(t)^2 \, dt
    }\,.
  \end{equation*}
  By Chebychev's inequality and the It\^{o} isometry, we see
  \begin{equation*}
    \P^{z_0}\paren{\hat K^c} \leq \frac{1}{\alpha^2}  = \frac{\P^{z_0}(K)}{2}\,,
  \end{equation*}
  and hence
  \begin{equation*}
    \P^{z_0}( K \cap \hat K ) \geq \frac{\P^{z_0} (K) }{2}\,.
  \end{equation*}
  Thus
  \begin{align}
    \nonumber
    \E^{z_0} \paren{ \one_{K} M_T }
     & \geq \E^{z_0}\paren[\Big]{ \one_{K \cap \hat K}
      \exp\paren[\Big]{
        -\alpha\paren[\Big]{\int_0^T \hat h(t)^2 \, dt}^{1/2}
        -\frac{1}{2} \int_0^T \hat h(t)^2 \, dt   }
    }
    \\
    \label{e:1kmt1}
     & \geq  \frac{\P^{z_0}(K)}{2}
      \inf_{K}
      \exp\paren[\Big]{
      -\alpha\paren[\Big]{\int_0^T \hat h(t)^2 \, dt}^{1/2}
      -\frac{1}{2} \int_0^T \hat h(t)^2 \, dt   }\,.
  \end{align}

  To estimate the exponential, note that on the event~$K$ we have
  \begin{align}
    \nonumber
    \abs{\hat h_i(t)} =
    \frac{\abs{h_i(t)}}{\sigma_i} & =
    \frac{A}{\sigma_i}
    \abs[\Big]{
      v_i(\tilde \gamma_t + \sigma B_t)
      - v_i(\tilde \gamma_t)
      + v_i(\tilde \gamma_t) - u_i(\tilde \gamma_t)
    }
    \\
    \label{e:hath1}
                                  & \leq \frac{\lambda \sqrt{A}}{\sigma_i} \sum_j
    \sigma_j \norm{\partial_j v_i}_{L^\infty(\tilde \Gamma + R)}
    + \frac{A \abs{u_i(\tilde \gamma_t) - v_i(\tilde \gamma_t)}}{\sigma_i}\,,
  \end{align}
  for every $i =1,2$.
 Combining~\eqref{e:hath1}  with~\eqref{e:1kmt1} completes the proof.
\end{proof}

\begin{lemma}\label{l:smallDev}
  Using the same notation as in Lemma~\ref{l:STPC}, we now additionally assume
  \begin{gather}
    \label{e:uvAssumption11}
    \max_{i \in \set{1, 2}}
    \sum_{j = 1,2} \frac{
      \sigma_j \norm{\partial_j v_i}_{L^\infty(R + \tilde \Gamma)}
    }{\sigma_i}
    \leq C_0
    \\
    \label{e:uvAssumption12}
    \sum_{i= 1,2} \int_0^T \frac{A^2 \abs{u_i(\tilde \gamma_t) - v_i(\tilde \gamma_t)}^2}{\sigma_i^2} \, dt \leq C_0^2\,.
  \end{gather}
  Then there exists $C_1 = C_1(C_0, \lambda, \beta) > 0$ such that
  \begin{equation*}
    \P^{z_0} \paren[\Big]{
      \sup_{0 \leq t \leq T}
      \abs{\sigma^{-1}( Z_t - \tilde \gamma_t )}_\infty  \leq \frac{\lambda}{\sqrt{A}} }
    \geq C_1
  \end{equation*}
\end{lemma}
\begin{proof}
  Following the proof of Lemma~\ref{l:STPC}, and using~\eqref{e:uvAssumption11}--\eqref{e:uvAssumption12} in~\eqref{e:hath1} gives
  \begin{equation*}
    \int_0^T \abs{\hat h(t)}^2 \, dt \leq 2 C_0^2 (1 + \lambda \beta d)\,.
  \end{equation*}
  Combined with~\eqref{e:1kmt1} the lemma follows.
\end{proof}

Next, we show the following estimate for the side boundary layer.

\begin{lemma} \label{l:STPC2}
  Let $z_0 \in \tilde\cH_n \defeq \cH_n - [c_0, 1-c_0]\times [0,1]$ and $n\in \N$;  $Z_t$ be a stochastic process satisfying~\eqref{eq:SDE}--\eqref{cond:sigma} and $\gamma_t$ be a deterministic process satisfying
  \begin{equation*}
    \partial_t \gamma_t = A v(\gamma_t) \quad \text{ with } \quad 
    \gamma_0 = z_0 \,.
  \end{equation*}
  Let~$T, R$ be as in~\eqref{e:TR}, and~$\Gamma = \set{\gamma(t) \st t \in [0, T] }$ be the image of~$\gamma$, and assume
  \begin{equation}\label{e:d1v2eq0}
    \partial_1 v_2 = 0 \qquad\text{in } \Gamma + R\,.
  \end{equation}
  For $M \geq 1$, let $\tilde R_\epsilon\subseteq [1-M/\sqrt{A}, 1+M/\sqrt{A}]$ be a Borel set,
  and $T = m/A$ for some $m \in \N$.
  Then, there exists a constant $C= C_{m, M}$ and $\epsilon_0 > 0$
  such that for all $\epsilon < \epsilon_0$,
  \begin{align} \label{ine:sideest}
    \nonumber
    \MoveEqLeft
    \P^{z_0}\paren[\Big]{
      \sup_{0 \leq t \leq T} \abs{Z_{1,t} -  \gamma_{1,t} } \leq  \frac{2M}{\sqrt{A}  } \,,
      \sup_{0\leq t\leq T}
      \abs{Z_{2,t} -  \gamma_{2,t} }\leq \frac{\epsilon}{\sqrt{A}}
      \,,
      Z_{1,T} - \gamma_{1,T} \in \tilde R_\epsilon
      }
    \\
     & \qquad\geq  
     C_{m,n}\P\paren[\Big]{ \abs{B_t} \leq \frac{2M}{\sqrt{A}  } \,,
      B_{1,T} \in \tilde R_\epsilon
    } 
  \end{align}
\end{lemma}
As before we write $Z = (Z_1, Z_2)$, $\gamma = (\gamma_1, \gamma_2)$, and the notation $Z_{i, t}$ and~$\gamma_{i, t}$ denotes the values of the coordinate processes~$Z_i$ and~$\gamma_i$ respectively at time~$t$.
\begin{proof}
  We follow the proof of Lemma~\ref{l:STPC}, and explicitly substitute~$\sigma_1 = 1$ and $\sigma_2 = \epsilon$.
  Our conclusion~\eqref{ine:sideest} will follow provided we can show
  \begin{equation}\label{e:hhatbd}
    \int_0^T \hat h(t)^2 \, dt \leq C\,,
  \end{equation}
  for some finite constant $C$, independent of $\epsilon$.
  To bound this, we use the upper bound~\eqref{e:hath1}, and observe that the second term on the right hand side is identically $0$ since~$u = v$.
  For the first term, the only term that may grow faster than $\sqrt{A}$ is when $i = 2$ and $j = 1$.
  In this case, the assumption~\eqref{e:d1v2eq0} guarantees that this term is identically $0$.
  Now squaring and integrating from $0$ to $T = m /A$ proves~\eqref{e:hhatbd} as desired.
\end{proof}
\begin{remark}\label{r:A51}
  If the velocity field~$v$ does not satisfy~\eqref{e:d1v2eq0}, then Lemma~\ref{l:STPC2} still holds provided $A$ is chosen so that $A \geq 1 / \epsilon^2$.
  To see this we note that~\eqref{e:hath1} implies
  \begin{equation*}
    \int_0^T \hat h(t)^2 \, dt \leq \frac{C m }{ A \epsilon^2}\,.
  \end{equation*}
  If $A \geq 1/\epsilon^2$ the right hand side of this is bounded independent of $\epsilon$, and so the remainder of the proof of Lemma~\ref{l:STPC2} remains unchanged.
\end{remark}

Finally, we prove Lemmas~\ref{l:optimalinteriorestbulk}, and Lemma~\ref{l:optimalinteriorestbulk}, which were used in the proofs of Theorem~\ref{t:energy} and Proposition~\ref{p:enstrophy}.
Both proofs follow along the lines of the above tube lemmas.

\begin{proof}[Proof of Lemma~\ref{l:tubecorner}]
  We only consider the case where $z_0 \in Q_0/2$. The other cases are similar.
  First, recall that, by a direct calculation, we can check $T \leq \abs{\ln\delta}/A$.
  Therefore, for small enough $\epsilon$, under the event
  $\set{  \abs{Z_{i, t} -  \gamma_{i,t}}  \leq  {\sigma_{i}}{\paren{\abs{\ln \delta}A}^{-1/2}} \,,~\forall t\leq T \,, i=1,2}$,
  we must have
  $Z_t \in Q_0$ for $t \leq T$. Thus,
  \begin{equation} \label{cornervelocity}
    v_1(Z_t) =  Z_{1, t} \quad \text{ and } \quad
    v_2(Z_t) = - Z_{2,t} \,.
  \end{equation}

  Now define
  \begin{gather*}
    d\tilde Z_t = A \begin{pmatrix} v_1( \gamma_t) \\ v_2(\gamma_t) \end{pmatrix}  \, dt + \sigma \, dB_t
  \end{gather*}
  and write
  \begin{equation} \label{cornervelocity2}
    h(t) \defeq A\begin{pmatrix} v_1( \gamma_t) - v_1(\tilde Z_t) \\
      v_2( \gamma_t) - v_2(\tilde Z_t)\end{pmatrix}
    =
    A\begin{pmatrix}  \gamma_{1,t} - \tilde Z_{1, t} \\
      - \gamma_{2, t} + \tilde Z_{2, t} \end{pmatrix}
    =
    A\begin{pmatrix}  -B_{1, t} \\
      \epsilon B_{2, t}\end{pmatrix}  \,.
  \end{equation}
  As before, we define $\hat h$ and a new measure $\hat \P$ by
  \begin{gather*}
    \hat h(t) \defeq \sigma^{-1} h(t) = \begin{pmatrix}  1 & 0 \\ 0 & 1/\epsilon \end{pmatrix} h(t)
    =  A\begin{pmatrix} -B_{1, t}\\ B_{2, t} \end{pmatrix}
    \,,
    \\
    d\hat \P = M_T \, d\P \,,
  \end{gather*}
  where
  \begin{equation*}
    M_t  \defeq \exp \paren[\Big]{ -\int_0^t \hat h(s) \, dB_s -  \frac{1}{2} \int_0^t \hat h(s)^2 \, ds   }\,,
  \end{equation*}
  for $0 \leq t \leq T$.
  By the Girsanov theorem, the process
  \begin{equation*}
    \hat B_t \defeq \int_0^t \hat h(s) \, ds +   B_t
  \end{equation*}
  is a Brownian motion with respect to the measure $\hat \P$.
  Therefore, by uniqueness of weak solutions of SDEs, we have
  \begin{align*}
    \E ( f(Z_t)) = \hat \E (f(\tilde Z_t))
     & = \hat \E ( f( \gamma_{1, t} + B_{1, t}, \gamma_{2, t} + \epsilon B_{2, t} ))   \\
     & = \E (f( \gamma_{1, t} +  B_{1, t}, \gamma_{2, t} + \epsilon B_{2, t}) M_t) \,.
  \end{align*}
  Hence
  \begin{align*}
    \MoveEqLeft
    \P^x \paren[\Big]{ \abs{Z_{i, t} -  \gamma_{i, t}} \leq \frac{\sigma_i}{\sqrt{\abs{\ln\delta}A } }\,,~\forall t\leq T \,, i = 1 ,2 }
    \\
     & = \E^x \paren[\Big]{ \one_{ \set[\big]{   \abs{ B_t }_\infty \leq  \paren{\abs{\ln\delta}A }^{-1/2}    \,,~\forall t\leq T}}
      M_T  }\,.
  \end{align*}

  Now, we have that, by It\^{o} formula,
  \begin{align*}
    \int_0^t \hat h(s) \, dB_s & = -A\int_0^t B_{1, s} \, dB_{1, s} + A\int_0^t B_{2, s} \, dB_{2, s} \\
                               & = \frac{A}{2}(-B_{1, t}^2 + B_{2, t}^2)  \,.
  \end{align*}
  Therefore,
  \begin{equation*}
    M_t \geq \exp\paren[\Big]{ -\frac{A}{2} ( B_{1, t}^2 + B_{2, t}^2 ) - A^2\int_0^t  (B_{1, s}^2 + B_{2, s}^2) \, ds } \,.
  \end{equation*}
  Therefore, as $T \leq \abs{\ln\delta}/A$, under the event
  \begin{equation*}
    K \defeq \set[\Big]{   \abs{B_t }_\infty
      \leq  \frac{1}{\sqrt{\abs{\ln\delta}A} }
      \,,~\forall t\leq T}  \,,
  \end{equation*}
  we must have
  \begin{equation*}
    M_T \geq  \exp\paren[\Big]{  -\frac{1}{2 \abs{\ln\delta}} - 2   } \geq C \,.
  \end{equation*}
  Since $\P(K) \approx 1/\abs{\ln\delta}^2$, this finishes the proof.
\end{proof}

\begin{proof}[Proof of Lemma~\ref{l:optimalinteriorestbulk}]
  Let $z_0 \in R_{h_0}$ and $T_0 = \inf\set{t > 0 \st \gamma_{2,t} \geq 1 - \delta}$, where~$\gamma$ is the solution to~\eqref{e:gamma} with~$\gamma_0 = z_0$. 
  A direct calculation shows that there exists $C_0$ for which $T_0 \leq C_0/A$.
  Furthermore, when $x_2 \in (0,1-2\delta)$, we have that
  \begin{equation*}
    v(x_1,x_2) = 
    \begin{pmatrix}
      \partial_2 H(x) \\
      \partial_1H
    \end{pmatrix} 
    = \begin{pmatrix}
       H_1(x_1) \\
      \pm x_2 
    \end{pmatrix} \,.
  \end{equation*}
  Therefore, following the proof of the tube lemma (Lemma~\ref{l:STPC}), we find that the function $\hat h(t)$ there satisfies
  \begin{equation*}
    \abs{\hat h(t) }
    = A \begin{pmatrix}
      \abs{H_1(\gamma_{1,t}) - H_1(\gamma_{1,t} + B_{1,t})} \\
       \abs{ B_{2,t}} 
    \end{pmatrix} \,.
  \end{equation*}
  Therefore, under the event
  $\set[\Big]{\sup_{t\leq T_0} \abs{B_t} \leq \sqrt{T_0} \,;
   B_{2,T_0} \geq 0}$, it is true that 
  \begin{equation}
    \label{ine:hhatbound}
    \int_0^{T_0} \abs{\hat h(t)}^2 \, dt \leq C  \,.
  \end{equation}
  We have that 
  \begin{equation*}
    K_1 \defeq \set[\Big]{\sup_{t\leq T_0}\abs{Z_t - \gamma_t} \leq \sqrt{T_0}\,; Z_{2,T_0} \geq 1-2\delta} \subseteq \set[\Big]{ \kappa_1 \leq \frac{C_0}{A}}  \,.
  \end{equation*}
 
  Following the proof of Lemma~\ref{l:STPC},
  by Girsanov's theorem and~\eqref{ine:hhatbound}, there exists $p_1 \in (0,1)$ such that 
  \begin{equation*}
  \P^{z_0}(K_1) \geq C \P \paren[\Big]{ \sup_{t\leq T_0}\abs{B_t} \leq \sqrt{T_0} \,; B_{2,T_0} \geq 0 } \geq p_1 \,,
  \end{equation*}
  from which~\eqref{ine:hittingtopprobabilitybulk} follows immediately.
  \end{proof} 

\bibliographystyle{halpha-abbrv}
\bibliography{refs,preprints}
\end{document}